\documentclass[a4paper,12pt,reqno]{amsart}
\usepackage[english]{babel}
\usepackage{amsmath}
\usepackage{amscd}
\usepackage{amsgen}
\usepackage{amsfonts}
\usepackage{amssymb}
\usepackage{amsthm}
\usepackage{comment}
\usepackage{stmaryrd}
\usepackage{amsmath}
\usepackage[all]{xy}
\usepackage{latexsym}
\usepackage{graphics}
\usepackage{verbatim}
\usepackage{url}
\usepackage{slashed}
\usepackage{mathabx}

\makeindex

\newtheorem{theorem}{Theorem}
\newtheorem{prop}[theorem]{Proposition}
\newtheorem{lem}[theorem]{Lemma}
\newtheorem{cor}[theorem]{Corollary}

\newtheorem{rem}[theorem]{Remark}
\newtheorem{example}[theorem]{Example}
\newtheorem{thm}{Theorem}

\newcommand{\lo}{\mathring{L}}               
\newcommand{\Hess}{\operatorname{Hess}}
\newcommand{\tr}{\operatorname{tr}}
\newcommand{\id}{\operatorname{Id}}      
\newcommand{\JF}{\mathcal{F}}                
\newcommand{\Fo}{\mathring{\mathcal{F}}}  
\newcommand{\W}{\overline{\mathcal{W}}}
\newcommand{\Ric}{\operatorname{Ric}}
\newcommand{\Curv}{\mathcal{R}}
\newcommand{\scal}{\operatorname{scal}}

\newcommand{\grad}{\operatorname{grad}}
\newcommand{\SC}{\mathcal{S}}                 
\newcommand{\res}{\operatorname{Res}}     
\newcommand{\G}{{\mathcal G}}                 
\newcommand{\LOP}{\mathcal {D}}               
\newcommand{\var}{{\textit{var}}}               

\def\J{{\sf J}}
\def\R{{\mathbb R}}
\def\Rho{{\sf P}}
\def\B{{\mathcal B}}                                  
\def\st{\stackrel{\text{def}}{=}}

\numberwithin{theorem}{section} \numberwithin{equation}{section}

\title{On singular Yamabe obstructions}

\author[Andreas Juhl and Bent {\O}rsted] {Andreas Juhl and Bent {\O}rsted}

\address{Department of Mathematics of {\AA}rhus University, Ny Munkegade 118,
8000 {\AA}rhus, Denmark}

\email{orsted@math.au.dk}

\address{Humboldt-Universit\"at, Institut f\"ur Mathematik, Unter den Linden 6, 10099 Berlin, Germany}

\email{ajuhl@math.hu-berlin.de}
\email{juhl.andreas@googlemail.com}


\begin{document}

\begin{abstract}
We discuss the singular Yamabe obstruction $\B_3$ of a hypersurface in a four-dimensional general background.
We derive various explicit formula for $\B_3$ from the original definition. We relate these formulas to corresponding
formulas in the literature. The proofs are elementary.
\end{abstract}

\maketitle

\begin{center} \today \end{center}

\tableofcontents

\section{Introduction}\label{intro}

On a smooth manifold equipped with a Riemannian metric, the basic objects for
geometry are the canonical Levi-Civita connection and the corresponding Riemann
curvature tensor. Conformal geometry is partly studying transformations preserving
angles, and partly polynomials in the curvature and its derivatives transforming
in simple ways under conformal changes of the metric, i.e., multiplying it by a
positive smooth function. Already the decomposition of the curvature into the Weyl
tensor, the scalar curvature, and the trace-free Einstein tensor indicates the role
of conformal geometry, since the Weyl tensor transforms just by multiplication by a
power of the conformal factor. H. Weyl himself was introducing the first gauge
theory in physics exactly via the local change of scale given by a conformal factor,
and related it to the Maxwell equations in relativistic field theory.
 \index{$\mathcal{W}_2$ \quad conformal Willmore functional}
Given a hypersurface in a Riemannian manifold one may in addition consider invariants 
coming from the embedding, i.e., not only by the metric on the hypersurface 
induced by the metric of the ambient space, but also from the normal geometry of the hypersurface. 
In the classical Gauss theory of surfaces in a three-dimensional Euclidean space, the product 
of the principal curvatures is the intrinsic Gauss/scalar curvature, whereas the arithmetic mean 
of the principal curvatures is the extrinsic mean curvature. A celebrated conformal invariant is 
the integral of the square of the mean curvature over the surface, the so-called Willmore energy, 
also relevant for the physical theory of surfaces. It is closely related to the conformally invariant integral 
$$       
   \mathcal{W}_2 = \int_{M^2} |\lo|^2 dvol
$$              
of the conformally invariant squared norm of the trace-free second fundamental form. In general, 
curvature invariants of a hypersurface consist of intrinsic invariants, coming from the induced metric, 
and extrinsic invariants, coming from the second fundamental form and the ambient metric. Whereas 
for a given manifold it is known how to describe the conformally invariant scalar curvature quantities 
using the Fefferman-Graham ambient metric \cite{FG-final}, an analogous description of scalar 
conformal invariants of a hypersurface is not known. Such a classification would also be of interest in physics 
\cite{Sol}.  

Recent years have seen attempts to embed the theory of the Feffermann-Graham ambient metric 
and of the related Poincar\'e-Einstein metric into a wider framework. For instance, Albin \cite{Albin} 
extended parts of the theory to Poincar\'e-Lovelock metrics including applications to $Q$-curvature. 
In another direction, Gover et al. \cite{Gover-AE,GW-announce,GW-LNY,GGHW} developed a tractor 
calculus approach to the problem of constructing higher-order generalizations of the Willmore functional 
$\mathcal{W}_2$. Here a central role is played by the singular Yamabe problem which replaces the Einstein 
condition. In \cite{ACF}, it was discovered that the obstruction to the {\em smooth} solvability of the 
singular Yamabe problem of a hypersurface of dimension $n$ is a scalar conformal invariant $\B_n$. 
The observation of Gover et al. that $\B_2=0$ is the Euler-Lagrange equation of $\mathcal{W}_2$ 
was the starting point of their theory. More generally, \cite{Graham-Yamabe} identified the 
equation $\B_n=0$ as the Euler-Lagrange equations of a conformally invariant 
functional which he termed the Yamabe energy. This energy is an analog of the integrated (critical) 
renormalized volume coefficient of a Poincar\'e-Einstein metric which in turn is related to the 
integrated (critical) Branson $Q$-curvature. Notably this connection to $Q$-curvature also extends to the 
present setting \cite{GW-reno}, \cite{JO}.

Formulas for the conformally invariant obstruction in terms of classical curvature data are not known for 
$n\ge 4$. But for $n=3$, such a formula for $\B_3$ was derived in \cite{GGHW} from a general tractor calculus 
formula in \cite{GW-LNY}. 

In the present paper, we shall take a classical perspective and derive formulas for $\B_3$ directly from 
its very definition. This approach is independent. We only apply standard linear algebra and tensor calculations. 
It confirms and partly corrects results in the literature. As technical tools we also employ some differential identities 
involving $L$. Partly these are classical such as those found by J. Simons, and partly these are less well-known. Finally, we 
apply classical style arguments to relate $\B_3$ to the variation of the conformally invariant functional  
$$
    \mathcal{W}_3 = \int_M (\tr(\lo^3) + (\lo,\W)) dvol
$$               \index{$\mathcal{W}_3$ \quad higher Willmore functional}
which can be viewed as a natural generalizations of the classical Willmore functional. This fits with Graham's 
theorem \cite[Theorem 3.1]{Graham-Yamabe}. Our arguments replace a technique introduced 
and exploited in \cite{GGHW} for the same purpose.

The formulation of the main result requires some notation. Let $L$ be the second fundamental form, $\lo$ its 
trace-free part and $H$ the mean curvature.  Let $\overline{W}$ be the Weyl tensor of the background metric. 
We also define two contractions $\overline{W}_{0}$ and $\W$ of $\overline{W}$ on $M$ 
by inserting a unit normal vector $\partial_0$ at the last and at the first and the last slot, respectively. We let 
the operator $\LOP$ act on trace-free symmetric bilinear forms $b$ on $M$ by 
$\LOP (b) = \delta \delta (b) + (\Rho,b)$. $\LOP$ maps trace-free symmetric bilinear forms to $C^\infty(M)$. 
It is conformally invariant in the sense that $e^{4\varphi} \hat{\LOP} (b) = \LOP (b)$ for $\varphi \in C^\infty(M)$. 
The Levi-Civita connections on $X$ and $M$ are denoted by $\bar{\nabla}$ and $\nabla$. For more details see 
Section \ref{not}. For the definition of the obstructions $\B_n$ we refer to Section \ref{SYP}.

\index{$\LOP$}

\begin{thm}\label{main1} Let $\iota: M^3 \hookrightarrow (X^4,g)$ be a smooth embedding. Then it holds
\begin{align}\label{B3-g-final}
   12 \B_3 & = 6 \LOP ((\lo^2)_\circ) + 2 |\lo|^4  + 2 \LOP (\W)  \notag \\
   & - 2 \lo^{ij} \bar{\nabla}^0(\overline{W})_{0ij0} - 4 \lo^{ij} \nabla^k \overline{W}_{kij0} - 4 H(\lo,\W)
  + 16 (\lo^2,\W)  + 4 |\W|^2 + 2 |\overline{W}_{0}|^2.
\end{align}
\end{thm}

For a conformally flat background, Theorem \ref{main1} reduces to the identity
\begin{equation}\label{B3-CF}
   6 \B_3  = 3 \LOP ((\lo^2)_\circ) + |\lo|^4 = \Delta (|\lo|^2) - |\nabla \lo|^2 + 3/2 |\delta(\lo)|^2 
   - 2 \J |\lo|^2 + |\lo|^4.
\end{equation}
The second equality follows from Lemma \ref{NEW3a}. We recall the well-known fact that for an odd-dimensional $M$, 
there are Poincar\'e-Einstein metrics $g_+$ such that the conformal compactification $g=r^2 g_+$ is smooth \cite{FG-final}.
For such a metric, it holds $\lo=0$, $\W = 0$ and even $\overline{W}_{0} = 0$ \cite[Proposition 4.3]{Gover-AE}. In 
particular, for $n=3$, the above formula confirms that $\B_3 = 0$. If $\lo=0$, then
$$
   6 \B_3 = \LOP(\W) + 2 |\W|^2 + |\overline{W}_{0}|^2.
$$
Formula \eqref{B3-g-final} confirms the conformal invariance of $\B_3$ (of weight $-4$). In fact, the conformal 
invariance of $\LOP$ implies that $\LOP((\lo^2)_\circ)$ and $\LOP(\W)$ are individually conformally invariant.
Furthermore, the sum of the first three terms in the second line of \eqref{B3-g-final} is conformally invariant. 
In fact, it holds
$$
  \lo^{ij} ( \bar{\nabla}^0(\overline{W})_{0ij0} + 2 \nabla^k \overline{W}_{kij0} +2 H \W_{ij}) 
   = (\lo, B) + (\lo^2,\W) + \lo^{ij} \lo^{kl} \overline{W}_{kijl}
$$
(see Lemma \ref{Bach-relation}) with a conformally invariant tensor $B$ of weight $-1$, i.e., $e^\varphi \hat{B} = B$, 
introduced in \cite[Lemma 2.1]{GGHW} and termed the hypersurface Bach tensor. All remaining terms in 
\eqref{B3-g-final} are individually conformally invariant.  

In the course of the proof of Theorem \ref{main1} we shall derive a number of equivalent formulas for $\B_3$ which 
are of interest in special cases. In particular, we find that
\begin{equation}\label{B3-flat-back}
   12 \B_3 = \Delta (|\lo|^2) + 6 (\lo,\Hess(H)) + 6  H \tr(\lo^3) + |\lo|^4 + 12 |dH|^2
\end{equation}
for a flat background (Corollary \ref{B3-inter-corr}). In \cite[Section 13.7]{JO}, we derived this formula in a 
different way as a consequence of a general expression for singular Yamabe obstructions \cite[Theorem 6]{JO}. 
One may also derive the general case of the above formula along this line. In \cite{JO}, we derived \eqref{B3-CF} by 
combining the conformal invariance of $\B_3$ with \eqref{B3-flat-back} and Simons identity.

The heat kernel asymptotics of elliptic boundary value problems for Laplace-type operators are 
another rich source of polynomials in the covariant derivatives of the curvature tensor and of the second fundamental 
form; in \cite{BG} explicit formulas are given for the first five (integrated) heat coefficients. It is a natural question to 
determine the polynomials of this nature which exhibit a conformal invariance. 

The paper is organized as follows. In Section \ref{second}, we derive identities for $\delta \delta (\lo^2)$ 
and $\Delta (|\lo|^2)$ which are crucial for later calculations and may also be of independent interest. They are 
closely related to some identities of J. Simons \cite{Simons}. In Section \ref{SYP} we define the singular 
Yamabe problem and the resulting obstructions in general dimensions. In Section \ref{B2-cl} we derive a 
formula for $\B_2$ and connect it with the Willmore equation. Section \ref{B3-general} is devoted to the 
derivation of formulas for $\B_3$ in terms of standard curvature quantities. The starting point will be a formula
in terms of the volume expansion of $g$ in geodesic normal coordinates. Along the way we derive 
several equivalent formulas for $\B_3$ which might be of interest under specific additional assumptions 
on the background metric and the embedding. The proof of the main theorem Theorem \ref{main1} is contained  
in the last subsections. It is here where we need the material of Section \ref{second}. In the final section, we 
derive the right-hand side of \eqref{B3-g-final} by variation of the functional $\mathcal{W}_3$ under normal 
variations of the embedding reproving a result in \cite{GGHW}, \cite{Graham-Yamabe}. Finally, we note that 
the main result is equivalent to \cite[Proposition 1.1]{GGHW} in the arXiv-version, but differs from its printed version - we 
clarify that issue in Remark \ref{GGHW-wrong}.

In view of the possible applications to physics, such as string and membrane theory, and the principles of holography, 
we have tried to be very explicit throughout.

\section{Notation}\label{not}

All manifolds $X$ are smooth. For a manifold $X$, $C^\infty(X)$ and $\Omega^p(X)$
denote the respective spaces of smooth functions and smooth $p$-forms. Let
$\mathfrak{X}(X)$ be the space smooth vector fields on $X$. Metrics on $X$ usually
are denoted by $g$. $dvol_g$ is the Riemannian volume element defined by $g$. The
Levi-Civita connection of $g$ is denoted by $\nabla_X^g$ or simply $\nabla_X$ for $X
\in \mathfrak{X}(X)$ if $g$ is understood. In these terms, the curvature tensor $R$
of the Riemannian manifold $(X,g)$ is defined by $R(X,Y)Z =\nabla_X \nabla_Y (Z) -
\nabla_Y \nabla_X (Z) - \nabla_{[X.Y]}(Z)$ for vector fields $X,Y,Z \in
\mathfrak{X}(X)$. The components of $R$ are defined by
$R(\partial_i,\partial_j)(\partial_k) = {R_{ijk}}^l \partial_l$. We also set
$\nabla_X (u) = \langle du,X \rangle$ for $X \in \mathfrak{X}(X)$ and $u \in
C^\infty(X)$. $\Ric$ and $\scal$ are the Ricci tensor and the scalar curvature of
$g$. On a manifold $(X,g)$ of dimension $n$, we set $2(n-1) \J = \scal$ and define
the Schouten tensor $\Rho$ of $g$ by $(n-2)\Rho = \Ric - \J g$. Let $W$ be the Weyl
tensor.

\index{$dvol_g$ \quad volume element of $g$}
\index{$\Omega^p$ \quad space of $p$-forms}
\index{$\nabla$ \quad Levi-Civita connection}
\index{$\grad (u)$ \quad gradient field of $u$}
\index{$\delta$ \quad divergence operator}
\index{$\Delta$ \quad Laplacian}

\index{$R$ \quad curvature tensor}
\index{$W$ \quad Weyl tensor}
\index{$\Ric$ \quad Ricci tensor}
\index{$\scal$ \quad scalar curvature}
\index{$\Rho$ \quad Schouten tensor}
\index{$\J$}

For a metric $g$ on $X$ and $u \in C^\infty(X)$, let $\grad_g(u)$ be the gradient of
$u$ with respect to $g$, i.e., it holds $g(\grad_g(u),V) = \langle du,V \rangle$ for
all vector fields $V \in \mathfrak{X}(X)$. $g$ defines pointwise scalar products
$(\cdot,\cdot)$ and norms $|\cdot|$ on $\mathfrak{X}(X)$, on forms $\Omega^*(X)$ and
on general tensors. Then $|\grad (u)|^2 = |du|^2$. In these definitions, we use the
metric as a subscript if needed for clarity. $\delta^g$ is the divergence operator
on differential forms or symmetric bilinear forms. On forms it coincides with the
negative adjoint $-d^*$ of the exterior differential $d$ with respect to the Hodge
scalar product defined by $g$. Let $\Delta_g = \delta^g d$ be the non-positive
Laplacian on $C^\infty(X)$. On the Euclidean space $\R^n$, it equals
$\sum_i \partial_i^2$. In addition, $\Delta$ will also denote the Bochner-Laplacian
(when acting on $L$).

A metric $g$ on a manifold $X$ with boundary $M$ induces a metric $h$ on $M$. In
such a setting, the curvature quantities of $g$ and $h$ will be distinguished by
adding a bar to those of $g$. In particular, the covariant derivative, the curvature
tensor and the Weyl tensor of $(X,g)$ are $\bar{\nabla}$, $\bar{R}$ and
$\overline{W}$. Similarly, $\overline{\Ric}$ and $\overline{\scal}$ are the Ricci
tensor and the scalar curvature of $g$.

\index{$\bar{\nabla}$ \quad Levi-Civita connection}
\index{$\overline{\Ric}_{0}$}
\index{$\overline{W}$ \quad Weyl tensor}
\index{$\overline{W}_0$}
\index{$\W$}
\index{$\overline{\scal}$ \quad scalar curvature}
\index{$\bar{\J}$}

A hypersurface usually is given by an embedding $\iota: M \hookrightarrow X$. Accordingly, 
tensors on $X$ are pulled back by $\iota^*$ to $M$. In formulas, we often omit this pull back.
For a hypersurface $\iota: M \hookrightarrow X$ with the induced metric $h = \iota^*(g)$ on $M$,
the second fundamental form $L$ is defined by $L(X,Y)= - h (\nabla^g_X(Y), N)$ 
for vector fields $X, Y \in \mathfrak{X}(M)$ and a unit normal vector field $\partial_0 = N$.
We set $n H = \tr_h(L)$ if $M$ has dimension $n$. $H$ is the mean curvature of $M$.
Let $\lo = L - H h$ be the trace-free part of $L$. We often identify $L$ with the shape operator $S$ 
defined by $h(X,S(Y)) = L(X,Y)$. 

We use metrics as usual to raise and lower indices. In particular, we set $(L^2)_{ij} = L_i^k L_{kj} 
= h^{lk} L_{il} L_{kj}$ and similarly for higher powers of $L$. We always sum over repeated indices.

The $1$-form $\overline{\Ric}_{0} \in \Omega^1(M)$ is defined by
$\overline{\Ric}_{0}(X) = \overline{\Ric}(X,\partial_0)$ for $X \in
\mathfrak{X}(M)$. Similarly, we write $b_0$ for the analogous $1$-form defined by a
bilinear form $b$ and we let $\overline{W}_0$ be the $3$-tensor on $M$ with components
by $\overline{W}_{ijk0}$, i.e., we always insert $\partial_0$ into the last slot. Moreover, we set
$\W_{ij} = \overline{W}_{0ij0}$. 

\index{$H$ \quad mean curvature}
\index{$\lo$ \quad trace-free part of $L$}
\index{$L$ \quad second fundamental form}
\index{$L^2$, $L^3$, $L^4$}
\index{$\partial_0$ \quad unit normal vector}

\section{Some second-order identities involving the second fundamental form}\label{second}

In the present section, we derive formulas for $\delta \delta (\lo^2)$ and $\Delta
(|\lo|^2)$ in terms of the geometry of the background and the intrinsic geometry of
$M$. The second of these formulas is closely related to a well-known formula of
Simons. The main results are Lemma \ref{NEW3a} and Lemma \ref{diff-key-g}. The
latter result will play an important role in Section \ref{B3-general}.

\begin{lem}\label{Id-basic} For $n=3$, it holds
\begin{equation}\label{Id-2}
   \delta \delta (\lo^2) = 2 \lo_{jk} \nabla^j \delta(\lo)^k + |\nabla \lo|^2 + \frac{1}{2} |\delta(\lo)|^2
   - \frac{1}{2} |\overline{W}_0|^2 + \kappa_1
\end{equation}
with
\begin{equation}\label{kappa-def}
   \kappa_1 \st (\nabla^i \nabla^j \lo^k_i - \nabla^j \nabla^i  \lo^k_i) \lo_{kj}.
\end{equation}
\end{lem}

\begin{proof} First, we calculate
\begin{align*}
    \delta \delta (\lo^2) & = \nabla^i \nabla^j \lo^2_{ij}
    = (\nabla^i \nabla^j \lo_i^k) \lo_{kj} + (\nabla^j \lo_i^k)(\nabla^i \lo_{kj})
   + (\nabla^i \lo_i^k) (\nabla^j \lo_{kj}) + \lo_i^k (\nabla^i \nabla^j \lo_{kj}) \\
   & =  (\nabla^i \nabla^j \lo^k_i) \lo_{kj} + (\nabla^j \lo_i^k)(\nabla^i \lo_{kj}) +
   \delta(\lo)^k \delta(\lo)_k + \lo_i^k \nabla^i \delta (\lo)_k.
\end{align*}
In the first term, we interchange covariant derivatives. This generates the curvature term
\begin{equation*}
    \kappa_1 \st (\nabla^i \nabla^j \lo_i^k) \lo_{kj} - (\nabla^j \nabla^i \lo_i^k) \lo_{kj}.
\end{equation*}
In the second term, we apply the Codazzi-Mainardi equation:
$$
   \nabla_j L_{ik} - \nabla_i L_{jk} = \bar{R}_{ijk0}.
$$
Its trace-free part gives
\begin{equation}\label{CM-trace-free}
   \nabla^j \lo^k_i - \nabla_i \lo^{jk} = - \frac{1}{2} \delta(\lo)^j h_i^k + \frac{1}{2} \delta(\lo)_i h^{jk}
   + {{\overline{W}_i}^{jk}}_0.
\end{equation}
In particular, we get
$$
   (\nabla^j \lo_i^k - \nabla_i \lo^{jk})  \nabla^i \lo_{jk} = - \frac{1}{2} (\delta(\lo),\delta(\lo))
  + {{\overline{W}_i}^{jk}}_0 \nabla^i \lo_{jk}.
$$
But
\begin{equation*}\label{add-1}
    \overline{W}_{ijk0} \nabla^i \lo^{jk}
    = \frac{1}{2} \left(\overline{W}_{ijk0} - \overline{W}_{jik0} \right) \nabla^i \lo^{jk}
    =  \frac{1}{2} \overline{W}_{ijk0} \left(\nabla^i \lo^{jk}  - \nabla^j \lo^{ki} \right)
    = -\frac{1}{2} |\overline{W}_0|^2,
\end{equation*}
where $|\overline{W}_0|^2 \st \overline{W}_{ijk0} \overline{W}^{ijk0}$. Thus
\begin{align*}
   (\nabla^j \lo_i^k) (\nabla^i \lo_{kj}) & = (\nabla^j \lo_i^k - \nabla_i \lo^{jk}) \nabla^i \lo_{kj}
  + (\nabla_i \lo^{jk})(\nabla^i \lo_{kj}) \\
  & = (\nabla_i \lo^{jk}) (\nabla^i \lo_{kj}) - \frac{1}{2} (\delta(\lo),\delta(\lo)) - \frac{1}{2} |\overline{W}_{0}|^2 \\
  & = (\nabla \lo, \nabla \lo) - \frac{1}{2} (\delta(\lo),\delta(\lo)) - \frac{1}{2} |\overline{W}_{0}|^2.
\end{align*}
These observations show that
\begin{equation}\label{dd-eval}
   \delta \delta (\lo^2) = 2 \lo_{jk} \nabla^j \delta(\lo)^k + |\nabla \lo|^2 + \frac{1}{2} |\delta(\lo)|^2
   - \frac{1}{2} |\overline{W}_{0}|^2 + \kappa_1
\end{equation}
with $\kappa_1$ as being defined in \eqref{kappa-def}.
\end{proof}

\begin{lem}\label{kappa-1a}
\begin{equation*}\label{K1}
   \kappa_1 = 3 (\lo^2,\Rho) + \J |\lo|^2.
\end{equation*}
\end{lem}

\begin{proof} By definition, we have
$$
   \kappa_1 = (\nabla^i \nabla^j \lo^k_i - \nabla^j \nabla^i  \lo^k_i) \lo_{kj}
   = \Curv^{ij} (\lo)_i^k \lo_{kj},
$$
where $\Curv$ denotes the curvature operator of $M$. We also observe that
$$
    \Curv^{ij} (\lo)_i^k \lo_{kj} = \Curv^{ij} (L)_i^k L_{kj}.
$$
Hence by
$$
   \Curv_{ij} (L)_{kl} = - L_l^{m} R_{ijkm} - L^m_k R_{ijlm}
$$
and the decomposition
\begin{equation}\label{KN}
   R_{ijkl} = -\Rho_{ik} h_{jl} + \Rho_{jk} h_{il} - \Rho_{jl} h_{ik} + \Rho_{il} h_{jk}
\end{equation}
(the Weyl tensor vanishes in dimension $3$) we get (see also Remark \ref{kappa1-2})
\begin{align}\label{kappa-curv}
   \kappa_1 & = - (L^m_i {R^{ij}}_{km} + L^{km} {{R^{ij}}_{im}}) L_{kj} \\
   & = L^m_i  (\Rho^i_k h^j_m - \Rho_k^j h_m^i + \Rho^j_m h_k^i - \Rho_m^i h_k^j) L_{kj}
   + L^{km} \Ric^j_m L_{kj} \notag \\
   & = 2 (L^2,\Rho) - 6 H (L,\Rho) + (L^2,\Ric) \notag \\
   & = 3 (L^2,\Rho) - 6 H (L,\Rho) + \J |L^2| \notag \\
   & = 3 (\lo^2,\Rho) + \J |\lo|^2. \notag
\end{align}
Alternatively, combining the Simons' identity \eqref{S-I} with the Gauss formula for the curvature
endomorphism yields the first identity in \eqref{kappa-curv}. This completes the proof.
\end{proof}

\begin{example}\label{kappa1-flat} For flat backgrounds, it holds
$$
   \kappa_1 = 3 H \tr(L^3) - |L|^4 = 3 H \tr(\lo^3) + 3 H^2 |\lo|^2 - |\lo|^4.
$$
\end{example}

\begin{proof} The Gauss identity gives
$$
    \J = -\frac{1}{4} |\lo|^2 + \frac{3}{2} H^2
$$
and the identity
\begin{equation}\label{Fial}
    \JF \st \iota^* \bar{\Rho} - \Rho + H \lo + \frac{1}{2} H^2 h \stackrel{!}{=} \lo^2 - \frac{1}{4} |\lo|^2 h + \W
\end{equation}
for the conformally invariant Fialkov tensor $\JF$ of weight $0$ \cite[Lemma 6.23.3]{J1} implies
$$
   \Rho = - \lo^2 + \frac{1}{4} |\lo|^2 h + H \lo + \frac{1}{2} H^2 h.
$$
Hence
$$
   3 (\lo^2,\Rho) + \J |\lo|^2 = - 3 \tr(\lo^4) + \frac{1}{2} |\lo|^4 + 3 H \tr(\lo^3) + 3 H^2 |\lo|^2
$$
and it suffices to apply the identity $2 \tr(\lo^4) = |\lo|^4$ (Corollary \ref{trace-id}).
\end{proof}

Now let
\begin{equation}\label{M2}
   \kappa_2 \st (\lo,\Delta (\lo)) - \frac{3}{2} \lo^{ij} \nabla_j \delta(\lo)_i.
\end{equation}

\begin{lem}\label{diff-simple}
\begin{equation*}
    \kappa_1 - \kappa_2 = \lo^{ij} \nabla^k \overline{W}_{kij0}.
\end{equation*}
\end{lem}

\begin{proof}
The trace-free part of the Codazzi-Mainardi equation reads
\begin{equation}\label{tf-CM}
   \nabla_i \lo_{kj} - \nabla_k \lo_{ij} - \frac{1}{2} \delta(\lo)_k h_{ij} + \frac{1}{2} \delta(\lo)_i h_{kj}
   = \overline{W}_{kij0}.
\end{equation}
Hence
$$
    \nabla^k \nabla_i \lo_{kj} - \nabla^k \nabla_k \lo_{ij} - \frac{1}{2} \nabla^k \delta(\lo)_k h_{ij}
   + \frac{1}{2} \nabla^k \delta(\lo)_i h_{kj} = \nabla^k \overline{W}_{kij0}.
$$
We commute the covariant derivatives in the first term and obtain
$$
   \lo^{ij} \nabla_i \delta(\lo)_j + \kappa_1 - (\lo, \Delta (\lo)) + \frac{1}{2} \lo^{ij} \nabla_j \delta(\lo)_i =
    \lo^{ij} \nabla^k \overline{W}_{kij0}.
$$
Therefore, we get
$$
   \frac{3}{2} \lo^{ij} \nabla_i \delta(\lo)_j - (\lo, \Delta(\lo)) + \kappa_1 = \lo^{ij} \nabla^k \overline{W}_{kij0}.
$$
In other words, we have
$$
   \kappa_1 - \kappa_2 =  \lo^{ij} \nabla^k \overline{W}_{kij0}.
$$
The proof is complete.
\end{proof}

One should compare this result with \cite[(2.12)]{GGHW}.

\begin{cor}\label{kappa-2-form}
\begin{equation}\label{K2}
    \kappa_2 = 3(\lo^2,\Rho) + \J |\lo|^2 - \lo^{ij} \nabla^k \overline{W}_{kij0}.
\end{equation}
\end{cor}

\begin{cor}\label{Laplace-L}
\begin{equation*}
    (\lo,\Delta(\lo)) = 3 (\lo,\Hess(H)) + 3 (\lo^2,\Rho) + \J |\lo|^2 + 3 \lo^{ij} \nabla_i (\bar{\Rho}_0)_j
    - \lo^{ij} \nabla^k \overline{W}_{kij0}.
\end{equation*}
\end{cor}

\begin{proof} We calculate
\begin{align*}
   (\lo,\Delta(\lo)) & = \frac{3}{2} \lo^{ij} \nabla_j  \delta(\lo)_i + \kappa_2 \qquad \mbox{(by \eqref{M2})} \\
   & = 3 (\lo,\Hess(H)) + 3 \lo^{ij} \nabla_j \bar{\Rho}_{0i} + \kappa_2 \qquad \mbox{(by Codazzi-Mainardi)}.
\end{align*}
Now we apply Corollary \ref{kappa-2-form}.
\end{proof}

Alternatively, we outline how Corollary \ref{Laplace-L} can be derived by using a Simons type formula.
First, we prove

\begin{lem}\label{pre-Simons} In general dimensions, it holds
$$
    \nabla_k \nabla_l (L)_{ij} = \nabla_i \nabla_j (L)_{kl} - \nabla_i \bar{R}_{kjl0}
    - \nabla_k \bar{R}_{lij0} + {{R_{ki}}^m}_l L_{mj} + {{R_{ki}}^m}_j L_{lm}.
$$
\end{lem}

\begin{proof} We start with the Codazzi-Mainardi equation
$$
    \nabla_i(L)_{lj} - \nabla_l (L)_{ij} = \bar{R}_{lij0}.
$$
Differentiation gives
$$
    \nabla_k \nabla_i (L)_{lj} - \nabla_k \nabla_l (L)_{ij} = \nabla_k \bar{R}_{lij0}.
$$
Now we commute the derivatives in the first term using
$$
     \nabla_k \nabla_i (L)_{lj} - \nabla_i \nabla_k (L)_{lj} = {{R_{ki}}^m}_l L_{mj} + {{R_{ki}}^m}_j L_{lm}.
$$
Hence
\begin{equation}\label{CM1}
   \nabla_i \nabla_k (L)_{lj} - \nabla_k \nabla_l (L)_{ij} = \nabla_k \bar{R}_{lij0}
   - {{R_{ki}}^m}_l L_{mj} - {{R_{ki}}^m}_j L_{lm}.
\end{equation}
Similarly, we differentiate the Codazzi-Mainardi equation
$$
   \nabla_j (L)_{kl} - \nabla_k (L)_{jl} = \bar{R}_{kjl0}
$$
and obtain
\begin{equation}\label{CM2}
   \nabla_i \nabla_j (L)_{kl} - \nabla_i \nabla_ k (L)_{jl} = \nabla_i \bar{R}_{kjl0}.
\end{equation}
Adding \eqref{CM1} and \eqref{CM2} proves the assertion.
\end{proof}

Lemma \ref{pre-Simons} implies
\begin{align*}
   (\nabla^i \nabla^j (L)_{ik} - \nabla^j \nabla^i (L)_{ik}) L^k_j
   & =  (\nabla^i \nabla^j (L)_{ik} - \nabla^k \nabla^i (L)_{ij}) L^k_j \quad \mbox{(by the symmetry of $L$)} \\
   & = R_{ikmj} L^{mi} L^{kj} + {R_{ikm}}^i L^{jm} L^k_j.
\end{align*}
One can easily check that this identity also holds if $L$ is replaced by $\lo$. This reproves \eqref{kappa-curv}.

Now taking a trace in Lemma \ref{pre-Simons} gives

\begin{lem}\label{pre-Simons-trace} In general dimensions, it holds
$$
   \Delta (L)_{ij} = n \Hess_{ij}(H) - \nabla^k \bar{R}_{kij0} + \nabla_i (\overline{\Ric}_0)_j
   + {{R_{ik}}^k}_m L_j^m - R_{kijm} L^{km}.
$$
\end{lem}

In particular, this gives a

\begin{proof}[Second proof of  Corollary \ref{Laplace-L}.]
Lemma \ref{pre-Simons-trace} and $(L,\Delta (L)) = (\lo,\Delta(\lo)) + 3 H \Delta (H)$ imply
\begin{align*}
   (\lo,\Delta(\lo)) & = 3 (\lo,\Hess (H)) + L^{ij} \nabla_i (\overline{\Ric}_0)_j
   - L^{ij} \nabla^k \bar{R}_{kij0} + L^{ij} {{R_{ik}}^k}_m  L_j^m - L^{ij} R_{kijm} L^{km} \\
   & =  3 (\lo,\Hess (H)) + (L^2)^{im} {{R_{ik}}^k}_m - L^{ij} L^{kl} R_{kijl} +  L^{ij} \nabla_i (\overline{\Ric}_0)_j
   - L^{ij} \nabla^k \bar{R}_{kij0}.
\end{align*}
Now
\begin{equation*}
    (L^2)^{im} {{R_{ik}}^k}_m = (L^2)^{ij} \Ric_{ij} = (L^2,\Rho) + \J |L|^2,
\end{equation*}
and \eqref{KN} implies
\begin{align*}
     L^{ij} L^{kl} R_{kijl} & = - 2 (L^2,\Rho) + 6 H (L,\Rho).
\end{align*}
Hence
\begin{align*}
   (\lo,\Delta(\lo)) = 3 (\lo,\Hess (H))  + 3 (L^2,\Rho) + \J |L|^2 - 6 H (L,\Rho) + 2 L^{ij} \nabla_i (\bar{\Rho}_0)_j
   - L^{ij} \nabla^k \bar{R}_{kij0}.
\end{align*}
In order to simplify this formula, we note that
$$
   3 (L^2,\Rho) + \J |L|^2 - 6 H (L,\Rho) = 3 (\lo^2,\Rho) + \J |\lo|^2
$$
and
$$
    \nabla^k \bar{R}_{kij0} = \nabla^k (\bar{\Rho}_0)_k h_{ij} - \nabla_j (\bar{\Rho}_0)_i + \nabla^k \overline{W}_{kij0}.
$$
Hence
\begin{align*}
   (\lo,\Delta(\lo)) = 3 (\lo,\Hess (H))  + 3 (\lo^2,\Rho) + \J |\lo|^2 + 3 L^{ij} \nabla_i (\bar{\Rho}_0)_j
   - 3H  \nabla^k (\bar{\Rho}_0)_k - L^{ij} \nabla^k \overline{W}_{kij0}.
\end{align*}
The last term is unchanged if we replace $L$ by $\lo$. This completes the proof.
\end{proof}

Combining Lemma \ref{pre-Simons} with some arguments using the Gauss formula relating the curvature tensors
of $X$ and $M$ leads to the following well-known identities which are due to Simons \cite{Simons, SchSY, HP,V}.

\begin{prop}\label{S-Id-g}
For any hypersurface $M^n \hookrightarrow X^{n+1}$ with the second fundamental form $L$, it holds
\begin{align}\label{S-I}
    \nabla_i \nabla_j (L)_{kl} & = \nabla_k \nabla_l (L)_{ij} + L_{ij} L^2_{kl} - L_{kl} L^2_{ij}
   + L_{il} L^2_{jk} - L_{jk} L^2_{il} \notag \\
    & - L_i^m \bar{R}_{jklm} - L_j^m \bar{R}_{iklm} + L_k^m \bar{R}_{lijm} + L_l^m \bar{R}_{kijm} \notag \\
    & + L_{ij} \bar{R}_{0kl0} - L_{kl} \bar{R}_{0ij0} + \bar{\nabla}_i (\bar{R})_{kjl0} + \bar{\nabla}_k (\bar{R})_{lij0}.
\end{align}
Taking a trace gives
\begin{align}\label{S-II}
   \Delta (L)_{ij} & = n \Hess_{ij}(H) + n H L^2_{ij} - L_{ij} |L|^2\notag \\
   & +  L^s_j \bar{R}_{ikks} + L_i^s \bar{R}_{jkks} - 2 L^{rs}  \bar{R}_{rijs} \notag \\
   & + n H \bar{R}_{0ij0} - L_{ij} \overline{\Ric}_{00} + \bar{\nabla}_k (\bar{R})_{ikj0} + \bar{\nabla}_i (\bar{R})_{jkk0}.
\end{align}
\end{prop}

For flat backgrounds, Proposition \ref{S-Id-g} specializes to

\begin{prop}\label{S-Id-flat}
For any hypersurface $M^n \hookrightarrow \R^{n+1}$ with the second fundamental form $L$, it holds
\begin{equation}\label{S-1}
   \nabla_i \nabla_j (L)_{kl} = \nabla_k \nabla_l(L)_{ij} + L_{ij} L^2_{kl} - L_{kl} L^2_{ij}
   + L_{il} L^2_{kj} - L_{kj} L^2_{il}.
\end{equation}
Hence
\begin{equation}\label{S-2}
   \Delta (L) = n  \Hess (H) +n H L^2 - L |L|^2
\end{equation}
and
\begin{equation}\label{S-2a}
   \frac{1}{2} \Delta (|L|^2) = n (L,\Hess (H)) + |\nabla L|^2 + n H \tr (L^3) - |L|^4.
\end{equation}
\end{prop}


\begin{rem}\label{kappa1-2}
The first part of Proposition \ref{S-Id-g} again confirms Lemma \ref{kappa-1a}. In fact, by the symmetry of $L$, we obtain
$$
   (\nabla^i \nabla^j (L)_{ki} - \nabla^j \nabla^i (L)_{ki}) L^k_j =  (\nabla^i \nabla^j (L)_{ki} - \nabla^k \nabla^i (L)_{ji}) L^k_j.
$$
In this identity one can replace $L$ by $\lo$. Hence \eqref{S-I} implies
$$
   \kappa_1 = (L^2)^{kl} \bar{R}_{kiil} - L^{il} L^{jk} \bar{R}_{kilj}.
$$
By the Gauss identity, we obtain
\begin{align*}
   \kappa_1 & =  3 H \tr(L^3) - |L|^4 + (L^2)^{kl} R_{kiil} - L^{il} L^{jk} R_{kilj} \\
   & +  (L^2)^{kl} (L_{ki} L_{il} - L_{kl} L_{ii}) - L^{il} L^{jk} (L_{kl} L_{ij} - L_{kj} L_{il}) \\
   &  = (L^2)^{kl} R_{kiil} - L^{il} L^{jk} R_{kilj}.
\end{align*}
The remaining arguments are as in the proof of Lemma \ref{kappa-1a}.
\end{rem}


\begin{rem}\label{Simons-flat}
For flat backgrounds, it holds $\kappa_2 = \kappa_1 = 3 H \tr(L^3) - |L|^4$ (Example \ref{kappa1-flat})
and the above results yield
\begin{equation*}\label{Id-1-flat}
    \frac{1}{2} \Delta (|\lo|^2) = 3 (\lo,\Hess(H)) + |\nabla \lo|^2 + 3 H \tr(L^3) - |L|^4
\end{equation*}
and
\begin{equation*}\label{Id-2-flat}
   \delta \delta (\lo^2) = 4 (\lo,\Hess(H)) + |\nabla \lo|^2 + 2 |dH|^2 + 3H \tr(L^3) - |L|^4.
\end{equation*}
As a consequence, we find the difference formula
\begin{equation}\label{basic-div}
   \frac{1}{2} \Delta (|\lo|^2) - \delta \delta (\lo^2) = - (\lo,\Hess(H)) - 2 |dH|^2.
\end{equation}
\end{rem}

Now combining Lemma \ref{Id-basic} with \eqref{M2}, we obtain
\begin{align*}
    \delta \delta (\lo^2) & = \frac{4}{3} (\lo,\Delta (\lo))
   + |\nabla \lo|^2+ \frac{1}{2} |\delta (\lo)|^2 - \frac{1}{2} |\overline{W}_{0}|^2 - \frac{4}{3} \kappa_2 + \kappa_1.
\end{align*}
Hence
\begin{align}\label{NEW}
   \delta \delta ((\lo^2)_\circ) & = \delta \delta (\lo^2) - \frac{1}{3} \Delta (|\lo|^2) \notag \\
   & = \delta \delta (\lo^2) - \frac{2}{3} |\nabla \lo|^2 - \frac{2}{3} (\lo,\Delta (\lo)) \notag \\
   & = \frac{2}{3} (\lo,\Delta (\lo)) + \frac{1}{3} |\nabla \lo|^2 + \frac{1}{2} |\delta (\lo)|^2
   - \frac{1}{2}|\overline{W}_{0}|^2 - \frac{1}{3} \kappa_1 + \frac{4}{3}(\kappa_1-\kappa_2).
\end{align}
Thus, using
$$
   (\Rho,(\lo^2)_\circ) = (\Rho,\lo^2) - \frac{1}{3} \J |\lo|^2,
$$
Lemma \ref{kappa-1a} and Lemma \ref{diff-simple}, we obtain

\begin{lem}\label{NEW3a}
\begin{align*}
  \delta \delta ((\lo^2)_\circ)  + (\Rho,(\lo^2)_\circ)
  & = \frac{2}{3} (\lo,\Delta \lo) + \frac{1}{3} |\nabla \lo|^2 + \frac{1}{2} |\delta (\lo)|^2 - \frac{2}{3} \J |\lo|^2 \\
  & -  \frac{4}{3}\lo^{ij} \nabla^k (\overline{W}_0)_{kij} - \frac{1}{2} |\overline{W}_{ikj0}|^2.
\end{align*}
\end{lem}

Lemma \ref{NEW3a} confirms \cite[Proposition 2.4]{GGHW} up to the sign of the term $|\overline{W}_{0}|^2$.

The following result extends the difference formula \eqref{basic-div} to general backgrounds. It will play an important
role in Section \ref{B3-general}.

\begin{lem}\label{diff-key-g} It holds
\begin{equation}\label{basic-diff-g}
    \Delta (|\lo|^2) - 2 \delta \delta (\lo^2)
    = - 2 (\lo,\Hess(H)) - 2(\lo,\nabla (\bar{\Rho}_0)) - |\delta(\lo)|^2
    - 2 \lo^{ij} \nabla^k \overline{W}_{kij0} + |\overline{W}_{0}|^2.
\end{equation}
\end{lem}

\begin{proof} We recall that
\begin{align*}\label{Simons-I}
   \Delta (|\lo|^2) & = 2 (\lo,\Delta(\lo)) + 2 |\nabla (\lo)|^2 \notag \\
   & = 6 (\lo,\Hess(H)) + 2 \kappa_1 + 6 (\lo,\nabla (\bar{\Rho}_{0}))
   - 2 \lo^{ij} \nabla^k \overline{W}_{kij0} + 2 |\nabla (\lo)|^2
\end{align*}
(by Lemma \ref{kappa-1a} and Lemma \ref{Laplace-L}) and
\begin{equation*}\label{Simons-II}
   2 \delta \delta (\lo^2) = 8 (\lo, \Hess(H)) + 8 (\lo,\nabla (\bar{\Rho}_0)) + 2 |\nabla (\lo)|^2
   + |\delta(\lo)|^2 - |\overline{W}_{0}|^2 + 2 \kappa_1
\end{equation*}
(by \eqref{dd-eval} and $\delta(\lo) = 2 dH + 2 \bar{\Rho}_0$ (Codazzi-Mainardi)). The difference of both
sums equals
\begin{equation}\label{diff-ex}
   -2 (\lo,\Hess(H)) - 2  (\lo,\nabla (\bar{\Rho}_0)) - |\delta(\lo)|^2 - 2 \lo^{ij} \nabla^k \overline{W}_{kij0}
  + |\overline{W}_{0}|^2.
\end{equation}
The proof is complete.
\end{proof}

Note that the left-hand side of \eqref{basic-diff-g} is a total divergence, i.e., integrates to $0$ on a closed $M$.
The fact that the sum of the first three terms on the right-hand side of \eqref{basic-diff-g} is a total divergence
follows by partial integration and the Codazzi-Mainardi equation. In fact, for closed $M$, we calculate 
\begin{align*}
    & \int_M -2 (\lo,\Hess(H)) - 2(\lo,\nabla (\bar{\Rho}_0)) - |\delta(\lo)|^2 dvol_h \\
    & = \int_M 2 (\delta(\lo),dH) + 2( \delta(\lo),\bar{\Rho}_0) - |\delta (\lo)|^2 dvol_h = 0
\end{align*}
by $2dH + \bar{\Rho}_0 = \delta(\lo)$. The fact that the additional terms on the right-hand side of \eqref{basic-diff-g}
also form a total divergence can be seen as follows. Partial integration gives
\begin{align*}
    - 2 \int_M \lo^{ij} \nabla^k \overline{W}_{kij0} dvol_h = 2 \int_M \nabla^k (\lo)^{ij} \overline{W}_{kij0} dvol_h.
\end{align*}
By the trace-free part of the Codazzi-Mainardi equation  \label{total}
$$
   \nabla_k (\lo)_{ij} - \nabla_i (\lo)_{kj} - \frac{1}{2} \delta(\lo)_i h_{kj} + \frac{1}{2} \delta(\lo)_k h_{ij}
   =  \overline{W}_{ikj0} = - \overline{W}_{kij0}
$$
and partial integration, this integral equals
\begin{align*}
   & 2 \int_M \nabla^i (\lo)^{kj} \overline{W}_{kij0} dvol_h - 2 \int_M \overline{W}^{kij0} \overline{W}_{kij0} dvol_h \\
   & = - 2 \int_M \lo^{kj} \nabla^i \overline{W}_{kij0} dvol_h -  2 \int_M |\overline{W}_{kij0}|^2 dvol_h
\end{align*}
Hence
$$
   \int_M (-4 \lo^{ij} \nabla^k \overline{W}_{kij0} + 2 |\overline{W}_{kij0}|^2)  dvol_h = 0.
$$
This proves the claim.

\section{The singular Yamabe problem and the obstruction}\label{SYP}

The material in this section rests on \cite{ACF} and \cite{GW-LNY}. 

Let $(X^{n+1},g)$ be a compact manifold with boundary $M$ of dimension $n$. The singular Yamabe problem asks
to find a defining function $\sigma$ of $M$ so that
\begin{equation}\label{syp}
   \scal (\sigma^{-2}g) = -n(n+1).
\end{equation}
The conformal transformation law of scalar curvature shows that
$$
    \scal(\sigma^{-2}g) = -n(n+1) |d\sigma|_g^2 + 2n \sigma \Delta_g(\sigma) + \sigma^2 \scal(g).
$$
Following \cite{GW-LNY}, we write this equation in the form
$$
   \scal(\sigma^{-2}g) = -n(n+1) \SC(g,\sigma),
$$
where                    \index{$\SC(g,\sigma)$} 
$$
   \SC(g,\sigma) \st |d\sigma|_g^2 + 2 \rho \sigma, \quad (n+1) \rho \st - \Delta_g(\sigma) - \sigma \J 
   \quad \mbox{and} \quad 2n \J = \scal(g).
$$
In these terms, $\sigma$ is a solution of \eqref{syp} iff $\SC(g,\sigma)=1$. Although such $\sigma$ exist
and are unique, in general, $\sigma$ is not smooth up to the boundary. The
smoothness is obstructed by a locally determined conformally invariant scalar
function on $M$ which is called the singular Yamabe obstruction.

In order to describe the structure of a solution $\sigma$ of the singular Yamabe problem more precisely, we
use geodesic normal coordinates. Let $r$ be the distance function of $M$ for the background metric $g$. Then
there are uniquely determined coefficients $\sigma_{(k)} \in C^\infty(M)$ for $2 \le k \le n+1$ so that the
smooth defining function
\begin{equation}\label{sigma-finite}
   \sigma_F \st  r + \sigma_{(2)} r^2 + \dots + \sigma_{(n+1)} r^{n+1}
\end{equation}
satisfies
\begin{equation}\label{Yamabe-finite}
   \SC(g,\sigma_F) = 1 + R r^{n+1}
\end{equation}
with a smooth remainder term $R$. The coefficients are recursively determined. In geodesic normal coordinates,
the metric $g$ takes the form $dr^2 + h_r$ with a one-parameter family $h_r$ of metrics on $M$. The condition
\eqref{Yamabe-finite} is equivalent to
$$
  |d\sigma_F|_g^2 - \frac{2}{n+1} \sigma_F \Delta_g(\sigma_F) - \frac{1}{n(n+1)} \sigma_F^2 \scal(g)
  = 1 +  R r^{n+1}.
$$
We write the left-hand side of this equation in the form
\begin{align}\label{Y-F}
  & \partial_r(\sigma_F)^2 + h_r^{ij} \partial_i (\sigma_F) \partial_j (\sigma_F) \notag \\
  & - \frac{2}{n+1} \sigma_F \left (\partial_r^2 (\sigma_F) + \frac{1}{2} \tr (h_r^{-1} h_r') \partial_r (\sigma_F)
  + \Delta_{h_r} (\sigma_F) \right) - \frac{1}{n(n+1)} \sigma_F^2 \scal(g)
\end{align}
and expand this sum into a Taylor series in the variable $r$. Then the vanishing of
the coefficient of $r^k$ for $k \le n$ is equivalent to an identity of the form
$$
   (k-1-n) \sigma_{(k+1)} = LOT,
$$
where $LOT$ involves only lower-order Taylor coefficients of $\sigma$. The latter relation
also indicates that there is a possible obstruction to the existence of an improved
solution $\sigma_F'$ which contains a term $\sigma_{(n+2)} r^{n+2}$ and satisfies
$\SC(g,\sigma_F') = 1 + R r^{n+2}$. Following \cite{ACF}, we define the
{\em singular Yamabe obstruction} by      \index{$\B_n$ \quad singular Yamabe obstruction}
\begin{equation}\label{B-def}
   \B_n \st \left( r^{-n-1} (\SC(g,\sigma_F) - 1) \right)|_{r=0} .
\end{equation}
Since $\sigma_F$ is determined by $g$, $\B_n$ is a functional of $g$. It is a key result that
$\B_n$ is a conformal invariant of $g$ of weight $-(n+1)$.  More precisely, we write $\hat{\B}_n$ 
for the obstruction defined by $\hat{g}=e^{2\varphi} g$ with $\varphi \in C^\infty(X)$. Then

\begin{lem}\label{B-CTL} $e^{(n+1) \iota^*(\varphi)} \hat{\B}_n = \B_n$.
\end{lem}

\index{$S_k$}

Let us be a bit more precise about the above algorithm. We set $S_k = \sum_{j=1}^k \sigma_{(j)}$
so that $S_{n+1} = \sigma_F$. Then the coefficients of $\sigma_F$ are recursively determined by the 
conditions
\begin{equation*}
   \SC(S_k)  = 1 + O(r^k).
\end{equation*}
More precisely, we recursively find
$$
   \SC(S_k) = 1 + r^{k-1} (c (n-k+2) \sigma_{(k)} + \cdots) + \cdots
$$
with $c = 2k/(n+1)$. Then the condition $\SC(S_k) -1 = O(r^k)$ with an unknown
coefficient $\sigma_{(k)}$ is satisfied iff the coefficient of $r^{k-1}$ in this expansion vanishes. This
can be solved for $\sigma_{(k)}$ if $k =2,3,\dots,n+1$. In the case $k=n+1$, we obtain
$$
   \SC(S_{n+1}) = 1 + O(r^{n+1})
$$
and the restriction of the latter remainder is the obstruction $\B_{n+1}$.

In the following, we shall need explicit formulas for the coefficients $\sigma_{(k)}$ for $k \le 4$.

First, we consider flat backgrounds.

We approximately solve the equation $\SC(g,\sigma_F) = 1$ for the flat metric $g$ by differentiation of the relation
\begin{align}\label{Y-F-flat}
  \partial_r(\sigma_F)^2 + h_r^{ij} \partial_i (\sigma_F) \partial_j (\sigma_F)
  - \frac{2}{n+1} \sigma_F \left (\partial_r^2 (\sigma_F) + \frac{1}{2} \tr (h_r^{-1} h_r') \partial_r (\sigma_F)
  + \Delta_{h_r} (\sigma_F) \right) = 0
\end{align}
in the variable $r$. Then, for general $n \ge 3$, we find the solution
\begin{equation}\label{sigma-F-flat}
   \sigma_F = r + \frac{r^2}{2} H - \frac{r^3}{3(n-1)} |\lo|^2 + r^4 \sigma_{(4)} + \cdots
\end{equation}
with the coefficient
\begin{equation}\label{sigma4-flat}
   \sigma_{(4)} = \frac{1}{24(n-2)} \left(6 \tr(\lo^3) + \frac{7n-11}{n-1} H |\lo|^2 + 3 \Delta(H)\right).
\end{equation}
Note that $\sigma_{(3)}$ is singular for $n=1$ and $\sigma_{(4)}$ is singular for $n=2$. In particular, we have
\begin{equation*}
   \sigma_F = r + \frac{r^2}{2} H - \frac{r^3}{6} |\lo|^2
   + \frac{r^4}{24} \left( 6 \tr (\lo^3) + 5 H |\lo|^2 + 3 \Delta (H) \right) + \cdots
\end{equation*}
if $n=3$ (\cite[(2.16)-(2.18]{GG}). These results are determined by the conditions
\begin{align*}
   \SC(S_2)  = 1 + O(r^2), \quad  \SC(S_3) = 1 + O(r^3), \quad \SC(S_4)  = 1 + O(r^4).
\end{align*}

In particular, the obstructions $\B_2$ and $\B_3$ are the restrictions of the remainder terms in
the second and the third expansions. More precisely, for $\B_2$ ($n=2$), we find
\begin{equation}\label{B2-new}
   \B_2 = (r^{-3}(\SC(S_3) -1))|_0 = -\frac{1}{3} (H |\lo|^2 + \Delta (H)) - \frac{2}{3} \tr(\lo^3).
\end{equation}
Since for $n=2$ the term $\tr(\lo^3)$ vanishes, we get
$$
   \B_2 = -\frac{1}{3} (H |\lo|^2 + \Delta(H)) .
$$

Similarly, for $n=3$, we get
\begin{align*}
   \B_3 & = (r^{-4}(\SC(S_4) -1))|_0 \\
   & = \frac{1}{12} (|\lo|^4 - 6 H \Delta(H) + \Delta (|\lo|^2) + 6 H \tr (\lo^3) + 3 |dH|^2 - 3 \Delta' (H)).
\end{align*}
where $\Delta_{h_r} = \Delta_h + r \Delta'_h + \cdots$. By the variation formula $\Delta'(u) = -2 (L,\Hess(u)) 
- 3 (dH,du)$ (see the proof of Lemma \ref{last-line}), this leads to
\begin{equation}\label{B3-flat}
   12 \B_3 = \Delta (|\lo|^2) + 6 (\lo,\Hess(H))  + |\lo|^4  + 6 H \tr(\lo^3) + 12 |dH|^2.
\end{equation}

For general background, we shall express the coefficients $\sigma_{(k)}$ in terms of
the volume coefficients $v_k$ of $h_r$, which are defined by the expansion
$$
   v(r) = \sum_{k\ge 0} r^k v_k
$$
of
$$
   v(r) \st  dvol(h_r)/dvol(h) = ( \det (h_r)/\det(h) )^{\frac{1}{2}}.
$$
This is convenient since the identity
\begin{equation}\label{trace-vol}
   \frac{v'(r)}{v(r)} = \frac{1}{2} \tr (h_r^{-1} h_r')
\end{equation}
provides natural formulas for the expansion of the coefficient $\tr (h_r^{-1} h_r')$ in \eqref{Y-F}.

Note that for the background $\R^{n+1}$ it holds
\begin{equation}\label{vol-flat}
   v(r) = \det (\id + r L)
\end{equation}
(see \cite[Section 3.4]{Gray}). Hence $v_{n+1} = 0$ in this case. For a general backgrounds, the volume coefficient
$v_{n+1}$ does not vanish, however.    \index{$v_k$ \quad volume coefficients}

The calculation of the remainder term in the expansion of $\SC(S_k)$ requires $k$ normal derivatives of the equation \eqref{Y-F-flat}.
Since the expansion of $S_k$ has a vanishing zeroth-order term, this amounts to take $k-1$ derivatives
of the trace term. In turn, this involves volume coefficients of the metric $h_r$ up to order $k$.  In general, we find
$$
   \B_n \st (r^{-(n+1)}(\SC(\sigma_F) -1))|_0 = (\cdots) - 2v_{n+1}.
$$
In particular, $\B_2$ involves the coefficient $v_3$ and $\B_3$ involves $v_4$.

Now, the above algorithm shows that, for general backgrounds and in general dimensions, the coefficients $\sigma_{(k)}$
are given by the formulas
\begin{align}\label{Y-sol-g}
    \sigma_{(2)} & = \frac{1}{2n}  v_1, \notag \\
    \sigma_{(3)} & =  \frac{2}{3(n-1)} v_2 - \frac{1}{3n} v_1^2 + \frac{1}{3(n-1)} \bar{\J}, \notag \\
    \sigma_{(4)} & = \frac{3}{4(n-2)} v_3 - \frac{9n^2-20n+7}{12n(n-1)(n-2)} v_1 v_2
    + \frac{6n^2-11n+1}{24n^2(n-2)} v_1^3 \notag \\
    & + \frac{2n-1}{6n(n-1)(n-2)} v_1 \bar{\J} + \frac{1}{4(n-2)} \bar{\J}' + \frac{1}{4(n-2)} \Delta (\sigma_{(2)}).
\end{align}
The observation that $\sigma_{(4)}$ has a (formal) pole at $n=2$ reflects the fact that there is no approximate solution
$\sigma_F$ up to order $r^4$ in that dimension. Similarly, $\sigma_{(5)}$ has a pole at $n=3$ - we shall not display an explicit
formula for $\sigma_{(5)}$, however. The obstruction to the existence of a smooth solution in $n=3$ is defined in terms of
$\SC(S_4)$.

In the flat case, the identity \eqref{vol-flat} implies that the volume coefficients $v_k$ are given by the elementary symmetric polynomials
$\sigma_k(L)$ in the eigenvalues of the shape operator defined by $L$. Hence Newton's formulas show that
\begin{align}\label{vol-flat-N}
   v_1 & = n H, \notag \\
   v_2 & = \frac{1}{2} (n H)^2  -\frac{1}{2} |L|^2, \notag \\
   v_3 & = \frac{1}{6} (n H)^3 - \frac{1}{2} n H |L|^2 + \frac{1}{3} \tr(L^3).
\end{align}
A combination of these formulas with \eqref{Y-sol-g} reproduces the expressions in \eqref{sigma-F-flat}, \eqref{sigma4-flat}.


\section{The singular Yamabe obstruction $\B_2$}\label{B2-cl}


In this section, we derive an explicit formula for the obstruction $\B_2$ from its definition in Section
\ref{SYP}. This reproves a result in \cite{ACF}. We also briefly recall the relation to the conformal Willmore
functional $\mathcal{W}_2$.

Let $n=2$. The formula for $\B_2$ in terms of volume coefficients reads
\begin{equation}\label{B2-new-g}
   \B_2 \st (r^{-3}(\SC(S_3) -1))|_0 = - 2 v_3 - \frac{1}{12} v_1^3 + \frac{1}{3} v_1 v_2
  - \frac{2}{3} \Delta (\sigma_{(2)}) - \frac{2}{3} v_1\bar{\J} - \frac{2}{3} \bar{\J}'.
\end{equation}

We recall that the term $v_3$ vanishes in $n=2$ in the flat case but not in the curved case. In the flat case,
this formula reduces to \eqref{B2-new}. The proof easily follows using $v_1 = 2H$, $v_2 = H^2 - |\lo|^2/2$
(see \eqref{vol-flat-N}) and $\sigma_{(2)} = H/2$. In the general case, the formulas for the volume coefficients in
Lemma \ref{v-coeff-3} imply
\begin{align*}
   \B_2 & = \left( \frac{1}{3} \bar{\nabla}_0(\overline{\Ric})_{00} - \frac{1}{6} \overline{\scal}' \right)
   - \frac{1}{3} H \overline{\scal} + H \overline{\Ric}_{00} - \frac{2}{3} (\lo,\bar{\G}) - \frac{1}{3} \Delta (H) - \frac{1}{3} H |\lo|^2
\end{align*}
using $\tr(\lo^3)=0$ in dimension $n=2$. Now the second Bianchi identity implies
$$
    \frac{1}{3} \bar{\nabla}_0(\overline{\Ric})_{00} - \frac{1}{6} \overline{\scal}' =
    - \frac{1}{3} \delta (\overline{\Rho}_0) + \frac{1}{3} (\lo,\bar{\Rho}) - H \overline{\Ric}_{00} + \frac{1}{3} H \overline{\scal}.
$$
(see \cite[(13.6.5)]{JO}). Hence using $(\lo,\bar{\G}) = (\lo,\bar{\Rho})$ we find
\begin{equation}\label{B2-g}
  \B_2 = -\frac{1}{3} ( \Delta (H) +  H |\lo|^2 + \delta (\overline{\Rho}_0) + (\lo,\overline{\Rho}) ).
\end{equation}
By Codazzi-Mainardi $dH = \delta(\lo) - \overline{\Ric}_0$, this formula is equivalent to
$$
   \B_2 = -\frac{1}{3} (\delta \delta (\lo) + H |\lo|^2 + (\lo,\bar{\Rho})).
$$
The latter formula for the obstruction was first derived in \cite[Theorem 1.3]{ACF}.

\begin{rem}\label{residue}
The coefficient $\sigma_{(4)}$ has a simple (formal) pole in $n=2$. Moreover, \eqref{Y-sol-g} implies
$$
   \res_{n=2}(\sigma_{(4)}) = \frac{3}{4} v_3 + \frac{1}{32} v_1^3 - \frac{1}{8} v_1 v_2
   + \frac{1}{4} v_1 \bar{\J} + \frac{1}{4} \bar{\J}' + \frac{1}{4} \Delta (\sigma_{(2)}).
$$
This formula shows the residue formula
\begin{equation*}\label{res-f2}
   \res_{n=2}(\sigma_{(4)}) = - \frac{3}{8} \B_2
\end{equation*}
being a special case of \cite[Lemma 16.3.9]{JO}).
\end{rem}

Let $K$ be the Gauss curvature of a surface $M \hookrightarrow \R^3$. By $2(H^2-K^2) = |\lo|^2$, the equation
$\Delta(H) + H|\lo|^2 = 0$ is equivalent to   
$$
   \Delta(H) + 2 H(H^2-K) = 0
$$
This equation is well-known as the Willmore equation for a surface $M$. It is the Euler-Lagrange equation 
of the Willmore functional
$$
   \mathcal{W}_2 = \int_M |\lo|^2 dvol_h
$$        
for variations of the embedding of $M$ \cite[Section 7.4]{Will}. In other words, 
$\B_2$ provides the Euler-Lagrange equation of $\mathcal{W}_2$. This fact extends to the curved case 
(for details see \cite[Section 13.9]{JO}).   \index{$\mathcal{W}_2$ \quad conformal Willmore functional}

\section{The singular Yamabe obstruction $\B_3$}\label{B3-general}

In this section, we determine explicit formulas for the obstruction $\B_3$. We shall start by expressing the 
definition \eqref{B-def} in terms of volume coefficients of the background metric and two normal derivatives 
of the scalar curvature. We simplify that result by repeated applications of the second Bianchi identity. A sequence 
of further transformations finally leads to Theorem \ref{main1}.  

\subsection{$\B_3$ in terms of volume coefficients}\label{B3-vol}

For $n=3$, the formulas in \eqref{Y-sol-g} read
\begin{align*}
  \sigma_{(2)} = \frac{1}{6} v_1, \quad \sigma_{(3)} = \frac{1}{9} (3 v_2 - v_1^2) + \frac{1}{6} \bar{\J}
\end{align*}
and
\begin{equation*}
  \sigma_{(4)} = \frac{1}{108} (81 v_3 -42 v_1 v_2 + 11 v_1^3) + \frac{5}{36} v_1 \bar{\J}
  + \frac{1}{4} \bar{\J}' + \frac{1}{4} \Delta (\sigma_{(2)}).
\end{equation*}
These quantities define $S_4$. We also recall the expansion $\Delta_{h_r} = \Delta_h + r \Delta_h' + \cdots$. In 
these terms, we obtain                               \index{$\Delta'$}

\begin{lem}\label{B3-volume} It holds
\begin{align}\label{B3-start}
   \B_3 \st (r^{-4}(\SC(S_4) -1))|_0
   & = - 2v_4 + \frac{1}{2} v_1 v_3 + \frac{1}{3} v_2^2 - \frac{7}{18} v_1^2 v_2 + \frac{2}{27} v_1^4 \notag \\
   & - \frac{1}{3} \bar{\J} v_2 - \frac{5}{12} \bar{\J}' v_1 - \frac{1}{4} \bar{\J}'' \notag \\
   & - \frac{1}{2} \Delta (\sigma_{(3)}) - \frac{1}{3} v_1 \Delta(\sigma_{(2)})
   - \frac{1}{2} \Delta' (\sigma_{(2)}) + |d\sigma_{(2)}|^2.
\end{align}
\end{lem}

This result follows by direct evaluation of the definition of $\B_3$. We omit the details.

In the remaining part of this section, we evaluate this formula.

First of all, we calculate the last line in \eqref{B3-start}.

\begin{lem}\label{last-line} It holds
\begin{align*}
  & - \frac{1}{2} \Delta (\sigma_{(3)}) - \frac{1}{3} v_1 \Delta(\sigma_{(2)}) - \frac{1}{2} \Delta' (\sigma_{(2)})
  + |d\sigma_{(2)}|^2 \\
  & = \frac{1}{12} \Delta (|\lo|^2) + \frac{1}{2} (\lo,\Hess(H)) + \frac{1}{6} \Delta (\bar{\Rho}_{00})
  + \frac{1}{2} (dH,\overline{\Ric}_0) + |dH|^2.
\end{align*}
\end{lem}

\begin{proof} We recall that $v_1 = 3 H$. By
\begin{align*}
    \sigma_{(2)} = \frac{1}{2} H \quad \mbox{and} \quad \sigma_{(3)} = \frac{1}{6} (-|\lo|^2 - 2 \bar{\Rho}_{00}),
\end{align*}
we obtain
\begin{align*}
  & - \frac{1}{2} \Delta (\sigma_{(3)}) - \frac{1}{3} v_1 \Delta(\sigma_{(2)})
  - \frac{1}{2} \Delta' (\sigma_{(2)}) + |d\sigma_{(2)}|^2 \\
  & =  \frac{1}{12} \Delta (|\lo|^2) + \frac{1}{6} \Delta (\bar{\Rho}_{00})  - \frac{1}{2} H \Delta H
  - \frac{1}{4} \Delta' (H) + \frac{1}{4} |dH|^2.
\end{align*}
Now the variation formula \cite[Proposition 1.184]{Besse}
\begin{align}\label{vDelta}
   (d/dt)|_0(\Delta_{g+th}(u)) = - (\nabla^g (du),h)_g - (\delta_g(h),du)_g + \frac{1}{2} (d (\tr_g(h)),du)_g
\end{align}
for the Laplacian implies (for $h = 2L$ and $g=h$)
$$
   \Delta' (u) = - 2 (L,\Hess(u)) - 2 (\delta(L), du) + 3 (dH,du).
$$
By Codazzi-Mainardi it holds $\delta(L) = 3dH + 2 \bar{\Rho}_0 = 3 dH + \overline{\Ric}_0$. Hence
$$
   \Delta'(u) = - 2 (L,\Hess(u)) - 3 (dH,du) - 2 (\overline{\Ric}_0,du).
$$
These results imply the assertion.
\end{proof}

\subsection{The volume coefficients}\label{vol-c}


The volume coefficients $v_j$ can be expressed in terms of the Taylor coefficients of $h_r$. These relations
follow by Taylor expansion of the identity \eqref{trace-vol} in the variable $r$ and solving the resulting relations for $v_j$.
We find
\begin{align*}
    2 v_1 & = \tr (h_{(1)}), \\
    8 v_2 & = \tr (h_{(1)})^2 + 4 \tr (h_{(2)}) - 2 \tr (h_{(1)}^2), \\
    48 v_3 & = \tr (h_{(1)})^3 + 12 \tr (h_{(1)}) \tr (h_{(2)}) + 24 \tr (h_{(3)})
    - 6 \tr(h_{(1)}) \tr (h_{(1)}^2) \\
    & - 24 \tr(h_{(1)} h_{(2)}) + 8 \tr (h_{(1)}^3)
\end{align*}
and
\begin{align*}
   384 v_4 & = \tr (h_{(1)})^4 + 24 \tr (h_{(1)})^2 \tr (h_{(2)}) + 48 \tr(h_{(2)})^2
  + 96 \tr(h_{(1)}) \tr (h_{(3)}) + 192 \tr (h_{(4)}) \\
   & - 12 \tr(h_{(1)})^2 \tr (h_{(1)}^2) - 48 \tr(h_{(2)}) \tr(h_{(1)}^2) + 12 \tr(h_{(1)}^2)^2
   - 96 \tr(h_{(1)}) \tr (h_{(1)} h_{(2)}) \\
   & - 192 \tr(h_{(1)} h_{(3)}) - 96 \tr(h_{(2)}^2) +32 \tr(h_{(1)}) \tr (h_{(1)}^3) + 192 \tr(h_{(1)}^2 h_{(2)})
  - 48 \tr(h_{(1)}^4.
\end{align*}
These formulas are valid in general dimension. In order to evaluate them, we apply the following results for
the coefficients $h_{(k)}$ for $k\le 3$  \cite{GG}, \cite[Proposition 13.2.1]{JO}.

\begin{lem}\label{h-coeff} In general dimensions, it holds
$$
    h_{(1)} = 2 L, \quad h_{(2)} = L^2 - \bar{\G} \quad \mbox{and} \quad
   3 (h_{(3)})_{ij}  =  -\bar{\nabla}_0 (\bar{R})_{0ij0} - 2 L_i^k \bar{\G}_{jk} - 2 L_j^k \bar{\G}_{ik},
$$
where $\bar{\G}_{ij} \st \bar{R}_{0ij0}$.    \index{$\bar{\G}$}
\end{lem}

As consequences, we find explicit formulas for the volume coefficients $v_k$ for $k \le 3$.

\begin{lem}\label{v-coeff-3} In general dimensions, it holds
\begin{align*}
     v_1 & = n H \\
     2 v_2 & = - \overline{\Ric}_{00} - |\lo|^2 + n(n\!-\!1) H^2 = \overline{\Ric}_{00} + \scal - \overline{\scal} \\
     6 v_3 & = - \bar{\nabla}_0(\overline{\Ric})_{00} + 2 (\lo,\bar{\G}) - (3n\!-\!2) H \overline{\Ric}_{00}
    + 2 \tr(\lo^3) - 3 (n\!-\!2) H |\lo|^2 + n(n\!-\!1)(n\!-\!2) H^3.
\end{align*}
\end{lem}

These formulas coincide with the corresponding terms in the expansion of the volume form in \cite[Theorem 3.4]{AGV}.
Note that this is obvious for $v_1$ and $v_2$ but requires to apply the Gauss identities
\begin{align*}
   \overline{\scal}-\scal  & = 2 \overline{\Ric}_{00} + |L|^2- n^2 H^2, \\
   \overline{\Ric} - \Ric & = \bar{\G} - n H L - L^2
\end{align*}
for $v_3$. Equivalent formulas can be found in \cite[Section 2]{GG}.

The coefficient $v_4$ is more involved. It depends on $h_{(k)}$ for $k\le 3$ and $\tr(h_{(4)})$. We shall not
discuss an explicit formula for $h_{(4)}$. For our purpose, it will be enough to prove the following formula for
the quantity $\tr(h_{(4)})$.

\begin{lem}\label{trace-h4} In general dimensions, it holds
\begin{equation}\label{trace-4}
    12 \tr (h_{(4)}) = - \bar{\nabla}_0^2 (\overline{\Ric})_{00} - 6 L^{ij} \bar{\nabla}_0 (\bar{R})_{0ij0}
   - 4 (L^2,\bar{\G}) + 4 (\bar{\G},\bar{\G}).
\end{equation}
\end{lem}

\begin{proof} A calculation of Christoffel symbols shows that
\begin{equation}\label{R-formula}
   \bar{R}_{0jk0} = \frac{1}{4} g^{ab} g_{aj}' g_{bk}' - \frac{1}{2} g_{jk}''
\end{equation}
\cite[(13.2.5)]{JO}. The assertion then follows by evaluating the second-order derivative in $r$ of this equation followed
by contraction with $h^{jk}$. Here are the details. Differentiating \eqref{R-formula} twice in $r$ at $r=0$
yields
\begin{align*}
   \partial_r^2 (\bar{R}_{0jk0}) & =\frac{1}{4} (g^{ab})'' g_{aj}' g_{bk}'
   + \frac{1}{4} g^{ab} (g_{aj})''' g_{bk}'  + \frac{1}{4} g^{ab} g_{aj}' (g_{bk})''' \\
   & + \frac{1}{2} (g^{ab})' (g_{aj})'' g_{bk}'
   + \frac{1}{2} (g^{ab})' (g_{aj})' (g_{bk})'' + \frac{1}{2} g^{ab} (g_{aj})'' (g_{bk})'' - \frac{1}{2} g_{jk}'''' \\
   & = 2 (3 L^2 + \bar{\G})^{ab} L_{aj} L_{bk} + 3 h^{ab} (h_{(3)})_{aj} L_{bk}
   + 3 h^{ab} L_{aj} (h_{(3)})_{bk} \\
   & - 4 L^{ab} (L^2-\bar{\G})_{aj} L_{bk} - 4 L^{ab} L_{aj} (L^2-\bar{\G})_{bk}
   + 2 h^{ab} (L^2- \bar{\G})_{aj} (L^2-\bar{\G})_{bk} - 12 (h_{(4)})_{jk} \\
   & =  2 (3 L^2 + \bar{\G})^{ab} L_{aj} L_{bk}
   + h^{ab} (-\bar{\nabla}_0(\bar{R})_{0aj0} - 2 (L \bar{\G})_{aj} - 2 (\bar{\G} L)_{aj}) L_{bk} \\
   & + h^{ab} L_{aj} (-\bar{\nabla}_0(\bar{R})_{0bk0} - 2 (L \bar{\G})_{bk} - 2 (\bar{\G} L)_{bk}) \\
   & - 4 L^{ab} (L^2-\bar{\G})_{aj} L_{bk} - 4 L^{ab} L_{aj} (L^2-\bar{\G})_{bk}
   + 2 h^{ab} (L^2- \bar{\G})_{aj} (L^2-\bar{\G})_{bk} - 12 (h_{(4)})_{jk}
\end{align*}
using $(h_r^{-1})_{ij} = h^{ij} - 2 L^{ij} r + (3(L^2)^{ij} + \bar{\G}^{ij}) r^2 + \cdots$. Hence
$$
   h^{jk} \partial_r^2(\bar{R}_{0jk0})
$$
equals
\begin{equation}\label{exp-1}
   - 2 L^{ij} \bar{\nabla}_0(\bar{R})_{0ij0} -2 (L^2,\bar{\G}) + 2 (\bar{\G},\bar{\G}) - 12 \tr (h_{(4)}).
\end{equation}
On the other hand, we calculate
\begin{align*}
   h^{jk} \partial_r^2  (\bar{R}_{0jk0}) & = \partial_r^2 ((h_r^{-1})_{jk} \bar{R}_{0jk0})
   - 2 (h_r^{-1})'_{jk} \partial_r (\bar{R}_{0jk0}) - (h_r^{-1})''_{jk} \bar{R}_{0jk0} \\
   & = \partial_r^2 (\overline{\Ric}_{00}) + 4 L^{jk} \partial_r (\bar{R}_{0jk0}) - 2(3 L^2 + \bar{\G},\G).
\end{align*}
Therefore, the relations
\begin{align*}
    \partial_r^2 (\overline{\Ric}_{00}) & = \bar{\nabla}_0^2(\overline{\Ric})_{00}, \\
    \partial_r(\bar{R}_{0jk0}) & = \bar{\nabla}_0(\bar{R})_{0jk0} + (L \bar{\G} + \bar{\G} L)_{jk}
\end{align*}
imply
\begin{equation}\label{exp-2}
  h^{jk} \partial_r^2  (\bar{R}_{0jk0}) = \bar{\nabla}_0^2(\overline{\Ric})_{00}
  + 4 L^{jk}  \bar{\nabla}_0(\bar{R})_{0jk0} + 8 (L^2,\bar{\G}) - 6 (L^2,\bar{\G}) - 2(\bar{\G},\bar{\G}).
\end{equation}
Now combining \eqref{exp-1} and \eqref{exp-2} proves the assertion.
\end{proof}

\begin{example}\label{PE-trace-test} Let $n \ge 3$ be general. Assume that $g = r^2 g_+$ for a
Poincar\'e-Einstein metric $g_+ = r^{-2} (dr^2 + (h - \Rho r^2  + \Rho^2 r^4/4))$
with conformally flat conformal infinity $h$. $\Rho$ is the Schouten tensor of $h$. $r$ is the distance in the metric
$g$ from the hypersurface $r=0$. The formula for $g$ shows that $\tr(h_{(4)}) = 1/4 |\Rho|^2$.
Comparing the coefficients of $r$ and $r^2$ in the expansions of $h_r$, shows that $L=0$ and $\bar{\G} = \Rho$.
Hence the above formula reduces to
$$
    12 \tr (h_{(4)}) = - \bar{\nabla}_0^2 (\overline{\Ric})_{00}  + 4(\Rho,\Rho)
    = - \partial_r^2 (\overline{\Ric}_{00}) + 4 (\Rho,\Rho).
$$
But \cite[Lemma 6.11.2]{J1} shows that $\bar{\Rho} = -1/(2r) \partial_r ({h}_r)$. Hence $\bar{\Rho}
= \Rho - 1/2 r^2 \Rho^2$. It follows that $\partial_r^2 (\bar{\Rho}) = - \Rho^2$. Therefore,
$\partial_r^2(\bar{\Rho}_{00}) = 0$ and we conclude that
$\partial_r^2(\overline{\Ric}_{00}) = \partial_r^2 (\bar{\J}) \stackrel{!}{=} |\Rho|^2$ (for $r=0$) using
\cite[Lemma 6.11.1]{J1}. Hence the right-hand side gives $3 |\Rho|^2$, i.e., we reproduced
the result $4 \tr(h_{4)}) = |\Rho|^2$. For general $h$, the Poincare\'e-Einstein metric also involves the Bach tensor.
But since the Bach tensor is trace-free, we still have $\tr(h_{(4)}) = 1/4 |\Rho|^2$ and we get the same conclusion.
\end{example}

The above results imply the following formula for $v_4$.

\begin{lem}\label{v4-form} It holds
\begin{align}\label{v4-L}
    24 v_4 & =  - \bar{\nabla}_0^2(\overline{\Ric})_{00}
    + 2 L^{ij} \bar{\nabla}_0 (\bar{R})_{0ij0} - 4 n H \bar{\nabla}_0 (\overline{\Ric})_{00} \notag \\
    & + 3 (\overline{\Ric}_{00})^2 - 2 (\bar{\G},\bar{\G}) + 8 n H (L,\bar{\G}) - 8 (L^2,\bar{\G})
    + 6 |\lo|^2 \overline{\Ric}_{00} - 6 n(n\!-\!1) H^2 \overline{\Ric}_{00}  \notag \\
    & + 24 \sigma_4(L)
\end{align}
or, equivalently,
\begin{align*}
    24 v_4 & =  -\bar{\nabla}_0^2(\overline{\Ric})_{00} + 2 \lo^{ij} \bar{\nabla}_0 (\bar{R})_{0ij0}
   - (4n\!-\!2) H \bar{\nabla}_0 (\overline{\Ric})_{00} \notag \\
   & + 3 (\overline{\Ric}_{00})^2 - 2 (\bar{\G},\bar{\G}) + 8(n\!-\!2) H (\lo,\bar{\G}) - 8 (\lo^2,\bar{\G}) \notag \\
   & - 2(n\!-\!1)(3n\!-\!4) H^2 \overline{\Ric}_{00} + 6 |\lo|^2 \overline{\Ric}_{00} \notag \\
   & + 24 \sigma_4(L).
\end{align*}
In particular, for $n=3$, we find
\begin{align}\label{v4-L-3}
    24 v_4 & =  - \bar{\nabla}_0^2(\overline{\Ric})_{00} + 2 \lo^{ij} \bar{\nabla}_0 (\bar{R})_{0ij0}
   - 10 H \bar{\nabla}_0 (\overline{\Ric})_{00} \notag \\
   & + 3 (\overline{\Ric}_{00})^2  - 2 (\bar{\G},\bar{\G}) + 8 H (\lo,\bar{\G})
   - 8 (\lo^2,\bar{\G}) - 20  H^2 \overline{\Ric}_{00} + 6 |\lo|^2 \overline{\Ric}_{00}
\end{align}
using $\sigma_4(L)=0$.
\end{lem}

\begin{proof}
This is a direct calculation. We omit the details.
\end{proof}

The three terms in the first line of \eqref{v4-L} coincide with the corresponding terms in the formula for
$v_4$ in \cite[Theorem 3.4]{AGV}. However, the remaining terms in both formulas are expressed in different
ways.

Note that Newton's identity
$$
   24 \sigma_4 (L) = \tr(L)^4 - 6 \tr(L)^2 |L|^2 + 3 |L|^4 + 8 \tr(L) \tr(L^3) - 6 \tr(L^4)
$$
gives
$$
   24 \sigma_4(L) = n(n-1)(n-2)(n-3) H^4 - 6 (n-2)(n-3) H |\lo|^2  + 8(n-3) H \tr(\lo^3) + 3 (|\lo|^4 - 2 \tr(\lo^4)).
$$

\begin{cor}\label{trace-id} Let $n=3$. Then $|\lo|^4 = 2 \tr(\lo^4)$.
\end{cor}

This result generalizes the fact that $\tr(\lo^3)=0$ for $n=2$.


\begin{example}
Assume that $g = r^2 g_+$ for a Poincar\'e-Einstein metric $g_+$ with conformal infinity $h$ (as in
Example \ref{PE-trace-test}). By $L=0$, the formula \eqref{v4-L} reads
\begin{align*}
    24 v_4  & =  - \bar{\nabla}_0^2(\overline{\Ric})_{00} + 3 (\overline{\Ric}_{00})^2 - 2 (\bar{\G},\bar{\G}).
\end{align*}
As above, we obtain
$$
   24 v_4 = -|\Rho|^2 + 3 \J^2 - 2 |\Rho|^2
$$
using the fact that $\bar{\G} = \Rho$ implies $\overline{\Ric}_{00} = \J$. This yields the well-known formula
$$
   v_4 = \frac{1}{8} (\J^2 - |\Rho|^2).
$$
\end{example}

\subsection{Evaluation I}\label{E1}

Now we combine the formula \eqref{B3-start} for $\B_3$ with the results in Section \ref{vol-c}.
A calculation yields the following result.

\begin{lem}\label{B3-inter-1} $12 \B_3$ equals the sum of
\begin{align}\label{B3-a}
   & \left(\bar{\nabla}_0^2(\overline{\Ric})_{00} - \frac{1}{2} \overline{\scal}''\right)  +
   5 H \left(\bar{\nabla}_0(\overline{\Ric})_{00} - \frac{1}{2} \overline{\scal}'\right)
   + 2 H \bar{\nabla}_0(\overline{\Ric})_{00} \notag \\
   & + 2 |\bar{\G}|^2  - 2 (\overline{\Ric}_{00})^2 + 2 \overline{\Ric}_{00} \bar{\J}
   + 8 H^2 \overline{\Ric}_{00} - 12 H^2 \bar{\J} + 6 (dH, \overline{\Ric}_0)
  + 2 \Delta (\overline{\Rho}_{00})  ,
\end{align}
\begin{equation}\label{B3-b}
   - 2 \lo^{ij} \bar{\nabla}_0(\bar{R})_{0ij0} - 2 H (\lo,\bar{\G}) + 8 (\lo^2,\bar{\G})
   - 4 |\lo|^2 \overline{\Ric}_{00}
   + 2 |\lo|^2 \bar{\J}
\end{equation}
and
\begin{equation}\label{B3-c}
   \Delta (|\lo|^2) + 6 (\lo,\Hess(H)) + 6  H \tr(\lo^3) + |\lo|^4 + 12 |dH|^2.
\end{equation}
\end{lem}

Since the terms in \eqref{B3-a} and \eqref{B3-b} vanish for the flat metric, we immediately reproduced
formula \eqref{B3-flat}. For later reference, we formulate that result as

\begin{cor}\label{B3-inter-corr}
For a hypersurface $M$ in the flat background $\R^4$, the obstruction $\B_3$ is given by
\begin{equation}\label{B3-flat-new}
   \Delta (|\lo|^2) + 6 (\lo,\Hess(H)) + 6  H \tr(\lo^3) + |\lo|^4 + 12 |dH|^2.
\end{equation}
\end{cor}

We continue with the discussion of the curved case.

Next, we simplify the sum \eqref{B3-a} using the second Bianchi identity. This step is analogous to the usage
of the second Bianchi identity in Section \ref{B2-cl}.       \index{$\bar{G}$ \quad Einstein tensor}

Let $\bar{G} \st \overline{\Ric} - \frac{1}{2} \overline{\scal} g$ be the Einstein tensor of $g$. The second
Bianchi identity implies $2 \delta^g (\overline{\Ric}) = d \overline{\scal}$. Hence
\begin{align*}
   & \bar{\nabla}_0(\overline{\Ric})(\partial_0,\partial_0) \\
   & = \delta^g(\overline{\Ric})(\partial_0) - g^{ij} \bar{\nabla}_{\partial_i}(\overline{\Ric})(\partial_j,\partial_0) \\
   & =  \frac{1}{2} \langle d \overline{\scal},\partial_0 \rangle - g^{ij} \partial_i (\overline{\Ric}(\partial_j,\partial_0))
   + g^{ij} \overline{\Ric} (\bar{\nabla}_{\partial_i}(\partial_j),\partial_0)
   + g^{ij} \overline{\Ric}(\partial_j,\bar{\nabla}_{\partial_i}(\partial_0)) \\
   & = \frac{1}{2} \langle d \overline{\scal},\partial_0 \rangle - h_r^{ij} \partial_i (\overline{\Ric}(\partial_j,\partial_0))
   + h_r^{ij} \overline{\Ric} (\nabla^{h_r}_{\partial_i}(\partial_j) - (L_r)_{ij} \partial_0,\partial_0)
   + h_r^{ij} \overline{\Ric}(\partial_j,\bar{\nabla}_{\partial_i}(\partial_0))  \\
   & = \frac{1}{2} \langle d \overline{\scal},\partial_0 \rangle - \delta^{h_r} (\overline{\Ric}_0)
   - n H_r \overline{\Ric}_{00} + h_r^{ij} \overline{\Ric}(\partial_j,\bar{\nabla}_{\partial_i}(\partial_0))
\end{align*}
on any level surface of $r$. Here $\delta^{h_r}$ denotes the divergence operator for the induced metric
on the level surfaces of $r$. Similarly, $L_r$ and $H_r$ are the second fundamental form and the
mean curvature of these level surfaces. Therefore, using $\bar{\nabla}_{\partial_i}(\partial_0)
= (L_r)_{ia} h_r^{ak} \partial_k$, we obtain
\begin{align*}
   \bar{\nabla}_0 (\bar{G})_{00} & = -\delta^{h_r} (\overline{\Ric}_0)
  - n H_r \overline{\Ric}_{00} + h_r^{ij} h_r^{ak} (L_r)_{ia} \overline{\Ric}_{jk},
\end{align*}
i.e., we have proved the relation
\begin{equation}\label{Bianchi}
    \bar{\nabla}_0(\bar{G})_{00}
    = - \delta^{h_r} (\overline{\Ric}_0) - n H_r \overline{\Ric}_{00} + (L_r,\overline{\Ric})_{h_r}
\end{equation}
on any level surface of $r$. Differentiating the identity \eqref{Bianchi} (for $n=3$)
with respect to $r$ at $r=0$ gives a formula for the term
$$
   \bar{\nabla}_0^2(\overline{\Ric})_{00} - \frac{1}{2} \overline{\scal}''
   = \bar{\nabla}_0^2 (\bar{G})_{00} = \partial_r (\bar{\nabla}_0 (\bar{G})_{00})
$$
for $r=0$. For that purpose, we use the variation formulas
\begin{align}\label{BV}
    3 H' & = - |L|^2 - \overline{\Ric}_{00}, \notag \\
    L' & = L^2 - \bar{\G}
\end{align}
for the variation of these quantities under the normal exponential map. Here we denote the derivative
in $r$ by a prime. We recall that in normal geodesic coordinates the metric $g$ takes the form $dr^2 +h_r$ with
$h_r = h + 2r L + \cdots$.  Moreover, let ${\delta}' \st (d/dr)|_0(\delta^{h_r})$. Then
\begin{align}\label{vdelta}
    {\delta}' (\omega) = - 2 (L, \nabla (\omega))_h -  2 (\delta (L),\omega)_h + 3 (dH,\omega)_h
\end{align}
for $\omega \in \Omega^1(M^3)$ \cite[(1.185]{Besse}. Note that the latter identity fits with the
variation formula \eqref{vDelta}. Now differentiating \eqref{Bianchi} (for $n=3$) implies
\begin{align*}
    \bar{\nabla}_0^2 (\bar{G})_{00} & = - {\delta}' (\overline{\Ric}_0) - \delta (\partial_r(\overline{\Ric}_0))
   - 3 {H}' \overline{\Ric}_{00} - 3 H \bar{\nabla}_0(\overline{\Ric})_{00} \\
    & + ({L}', \overline{\Ric}) + L^{ij} \partial_r (\overline{\Ric}_{ij}) - 4 (L^2,\overline{\Ric}),
\end{align*}
where $L' = (d/dr)|_0(L_r)$.  Note that the last term comes from the differentiation of $h_r$. Hence
\begin{align*}
    \bar{\nabla}_0^2 (\bar{G})_{00} & = 2 (L, \nabla (\overline{\Ric}_0)) + 2 (\delta(L), \overline{\Ric}_0)
   - 3 (dH,\overline{\Ric}_0) \notag \\
   & - \delta (\bar{\nabla}_0(\overline{\Ric})_{0}) - \delta ((L \overline{\Ric})_{0})
   + |L|^2  \overline{\Ric}_{00} + (\overline{\Ric}_{00})^2 - 3 H \bar{\nabla}_0 (\overline{\Ric})_{00} \notag \\
   & - 3 (L^2, \overline{\Ric}) - (\bar{\G}, \overline{\Ric}) + (L, \bar{\nabla}_0(\overline{\Ric}))
  + 2 (L^2, \overline{\Ric})
\end{align*}
at $r=0$. Here we used the relations
\begin{align*}
   \bar{\nabla}_0(\overline{\Ric}_0) = \partial_r (\overline{\Ric}_{00})  - (L \overline{\Ric})_{0} \quad \mbox{and} \quad 
   \bar{\nabla}_0(\overline{\Ric})_{ij} = \partial_r (\overline{\Ric}_{ij}) - (L \overline{\Ric} + \overline{\Ric} L)_{ij}.
\end{align*}
Now, separating the trace-free part of $L$ in some terms, we obtain
\begin{align*}
    \bar{\nabla}_0^2 (\bar{G})_{00} & = 2 (\lo,\nabla (\overline{\Ric}_0)) + 2 H \delta (\overline{\Ric}_0)
   + 2 (\delta(\lo),\overline{\Ric}_0) - (dH,\overline{\Ric}_0) \\
   & - \delta (\bar{\nabla}_0(\overline{\Ric})_{0}) - \delta ((\lo \overline{\Ric})_{0}) - \delta (H \overline{\Ric}_0) \\
   & + |L|^2  \overline{\Ric}_{00} + (\overline{\Ric}_{00})^2 - 3 H \bar{\nabla}_0 (\overline{\Ric})_{00} \\
   & - (L^2, \overline{\Ric}) - (\bar{\G}, \overline{\Ric})
   + (\lo, \bar{\nabla}_0(\overline{\Ric})) + H \overline{\scal}' - H \bar{\nabla}_0(\overline{\Ric})_{00}.
\end{align*}
This leads to the following result.

\begin{lem}\label{Nabla-2G} It holds
\begin{align*}
     \bar{\nabla}_0^2 (\bar{G})_{00} =  & - 4  H \bar{\nabla}_0 (\overline{\Ric})_{00} + H \overline{\scal}' \\
    & + 2 (\lo, \nabla (\overline{\Ric}_0))
    - \delta (\bar{\nabla}_0(\overline{\Ric})_{0}) + (\lo, \bar{\nabla}_0(\overline{\Ric})) \\
    & +  H \delta (\overline{\Ric}_0) - 2 (dH,\overline{\Ric}_0) + 2 (\delta(\lo),\overline{\Ric}_0)
    - \delta ((\lo \overline{\Ric})_{0}) \\
    & + |L|^2  \overline{\Ric}_{00} - (L^2, \overline{\Ric}) + (\overline{\Ric}_{00})^2 - (\bar{\G}, \overline{\Ric}).
\end{align*}
\end{lem}

Lemma \ref{Nabla-2G} enables us to replace second-order normal derivatives in Lemma \ref{B3-inter-1} by first-order
normal and tangential derivatives. More precisely, it follows that \eqref{B3-a} equals
\begin{align*}
    & 3 H \bar{\nabla}_0(\bar{G})_{00}
   + 2 (\lo, \nabla (\overline{\Ric}_0)) - \delta (\bar{\nabla}_0(\overline{\Ric})_{0})
   + (\lo, \bar{\nabla}_0(\overline{\Ric})) \\
   & +  H \delta (\overline{\Ric}_0) + 4 (dH,\overline{\Ric}_0) + 2 (\delta(\lo),\overline{\Ric}_0)
    - \delta ((\lo \overline{\Ric})_{0}) \\
   & + |L|^2  \overline{\Ric}_{00} - (L^2, \overline{\Ric}) + (\overline{\Ric}_{00})^2 - (\bar{\G}, \overline{\Ric}) \\
   & + 2 |\bar{\G}|^2 - 2 (\overline{\Ric}_{00})^2 + 2 \overline{\Ric}_{00} \bar{\J}
   + 2 \Delta (\overline{\Rho}_{00}) + 8 H^2 \overline{\Ric}_{00} - 12 H^2 \bar{\J}.
\end{align*}

Now a second application of the second Bianchi identity \eqref{Bianchi} enables us to replace the first-order normal derivative
of the Einstein tensor in this formula by tangential derivatives. Hence \eqref{B3-a} equals the sum
\begin{align*}
   & -3 H \delta (\overline{\Ric}_0) - 9 H^2 \overline{\Ric}_{00} + 3 H (L,\overline{\Ric})  \\
   &  + 2 (\lo, \nabla (\overline{\Ric}_0))) - \delta (\bar{\nabla}_0(\overline{\Ric})_{0})
   + (\lo, \bar{\nabla}_0(\overline{\Ric})) \\
   & +  H \delta (\overline{\Ric}_0) + 4 (dH,\overline{\Ric}_0) + 2 (\delta(\lo),\overline{\Ric}_0)
    - \delta ((\lo \overline{\Ric})_{0}) \\
   & + |L|^2  \overline{\Ric}_{00} - (L^2, \overline{\Ric}) + (\overline{\Ric}_{00})^2
   - (\bar{\G}, \overline{\Ric}) \\
   & + 2 |\bar{\G}|^2 - 2 (\overline{\Ric}_{00})^2
   + 2 \overline{\Ric}_{00} \bar{\J} + 2 \Delta (\overline{\Rho}_{00}) + 8 H^2 \overline{\Ric}_{00}
   - 12 H^2 \bar{\J}.
\end{align*}
A slight reordering and simplification of this sum shows that the sum \eqref{B3-a} equals
\begin{align}\label{H1}
   & - \delta (\bar{\nabla}_0(\overline{\Ric})_{0}) - 2 H  \delta (\overline{\Ric}_0)
  +4  (dH, \overline{\Ric}_0)  \notag \\
   & + (\lo, \bar{\nabla}_0(\overline{\Ric}))
  + 2 (\lo, \nabla (\overline{\Ric}_0)) + 2 (\delta(\lo),\overline{\Ric}_0) - \delta ((\lo \overline{\Ric})_{0})  \notag \\
  & + |L|^2  \overline{\Ric}_{00} - (L^2, \overline{\Ric})  + 3 H (L,\overline{\Ric}) \notag \\
  &- (\overline{\Ric}_{00})^2 - (\bar{\G}, \overline{\Ric})  + 2 |\bar{\G}|^2
   + 2 \overline{\Ric}_{00} \bar{\J} + 2 \Delta (\overline{\Rho}_{00}) - H^2  \overline{\Ric}_{00}  - 12 H^2 \bar{\J} .
\end{align}

We continue by further simplifying the sum \eqref{H1}. First of all, we observe

\begin{lem}\label{van-term} Let $n=3$. Then
$$
   - (\overline{\Ric}_{00})^2 - (\bar{\G}, \overline{\Ric})  + 2 |\bar{\G}|^2
   + 2  \overline{\Ric}_{00} \bar{\J} = 2 (\bar{\Rho},\W) + 2 |\W|^2.
$$
\end{lem}

\begin{proof} We recall that $\bar{\G} _{ij} = \bar{\Rho}_{ij} + \bar{\Rho}_{00} h_{ij} + \W_{ij}$.
Therefore, we get $\overline{\Ric}_{ij} = 2 \bar{\Rho}_{ij} + \bar{\J} h_{ij}
= 2 \bar{\G}_{ij} - 2 \W_{ij} - (2 \bar{\Rho}_{00} - \bar{\J}) h_{ij}$. Thus
$
   (\bar{\G},\overline{\Ric}) = 2 (\bar{\G},\bar{\G}) - 2 (\bar{\G},\W) - (2 \bar{\Rho}_{00} - \bar{\J})
   \overline{\Ric}_{00}.
$
This relation implies
$$
   2 |\bar{\G}|^2 - (\bar{\G}, \overline{\Ric})
  = (2 \bar{\Rho}_{00}  -\bar{\J})  \overline{\Ric}_{00} + 2 (\bar{\G},\W)  =
   (\overline{\Ric}_{00})^2 - 2\overline{\Ric}_{00} \bar{\J} + 2 (\bar{\Rho},\W)  + 2 |\W|^2.
$$
The proof is complete.
\end{proof}

Next, we have the following identities.

\begin{lem}\label{help-2} Let $n=3$. Then
$$
   |L|^2 \overline{\Ric}_{00} - (L^2, \overline{\Ric}) = |\lo|^2 \overline{\Ric}_{00} - (\lo^2,\overline{\Ric})
   - 2H (\lo,\overline{\Ric}) + 4 H^2 \overline{\Ric}_{00} - 6 H^2 \bar{\J}
$$
and
$$
   (L,\overline{\Ric}) = (\lo,\overline{\Ric}) + 6 H \bar{\J} - H \overline{\Ric}_{00}.
$$
\end{lem}

\begin{proof} The assertions follow by direct calculation.
\end{proof}

By Lemma \ref{van-term} and Lemma \ref{help-2}, the last two lines of \eqref{H1} simplify to
\begin{align*}
   |\lo|^2 \overline{\Ric}_{00} - (\lo^2,\overline{\Ric}) + H (\lo,\overline{\Ric}) + 2 \Delta (\overline{\Rho}_{00})
   + 2(\bar{\Rho},\W) + 2 |\W|^2.
\end{align*}
Therefore, \eqref{B3-a} equals
\begin{align*}\label{H1a}
    & - \delta (\bar{\nabla}_0(\overline{\Ric})_{0}) - 2 H  \delta (\overline{\Ric}_0) + 4 (dH, \overline{\Ric}_0)
   \notag \\
    & + (\lo, \bar{\nabla}_0(\overline{\Ric}))
   + 2 (\lo, \nabla (\overline{\Ric}_0)) + 2 (\delta(\lo),\overline{\Ric}_0) - \delta ((\lo \overline{\Ric})_{0})  \notag \\
    & +|\lo|^2 \overline{\Ric}_{00} - (\lo^2,\overline{\Ric}) + H (\lo,\overline{\Ric}) + 2 \Delta (\overline{\Rho}_{00})
   + 2(\bar{\Rho},\W) + 2 |\W|^2.
\end{align*}
Hence Lemma \ref{B3-inter-1} implies

\begin{lem}\label{B3-inter2}
$12 \B_3$ equals the sum of
\begin{equation}\label{B3-g}
   - \delta (\bar{\nabla}_0(\overline{\Ric})_{0}) - 2 \delta ( H \overline{\Ric}_0) + 6 (dH,\overline{\Ric}_0)
   + 2 \Delta (\overline{\Rho}_{00}),
\end{equation}
\begin{align}\label{B3-u}
   & (\lo, \bar{\nabla}_0(\overline{\Ric})) - 2 \lo^{ij} \bar{\nabla}_0(\bar{R})_{0ij0}
   + 2 (\lo, \nabla (\overline{\Ric}_0)) + 2 (\delta(\lo),\overline{\Ric}_0) - \delta ((\lo \overline{\Ric})_{0})  \notag \\
   & -3 |\lo|^2 \overline{\Ric}_{00} - (\lo^2,\overline{\Ric}) + H (\lo,\overline{\Ric}),
\end{align}
\begin{equation}\label{B3-gf}
   - 2 H (\lo,\bar{\G}) + 8 (\lo^2,\bar{\G}) + 2 |\lo|^2 \bar{\J},
\end{equation}
\begin{equation}\label{Weyl}
   2(\bar{\Rho},\W) + 2 |\W|^2
\end{equation}
and the flat terms
\begin{align}\label{B3-gc}
   & 6 (\lo,\Hess(H)) + \Delta (|\lo|^2) + 6 H \tr(\lo^3) + |\lo|^4 + 12 |dH|^2.
\end{align}
\end{lem}

Note that the first term in \eqref{B3-g} contains a normal derivative of $\overline{\Ric}$. Likewise the first two terms
in \eqref{B3-u} contain normal derivatives of the curvature of $g$. All other terms in
\eqref{B3-g}--\eqref{B3-gc} live on $M$. The {\em mixed} terms in \eqref{B3-u} and \eqref{B3-gf} involve the curvature of
$g$ and $L$. Finally, the terms in \eqref{Weyl} involve the Weyl tensor, and the terms in \eqref{B3-gc} are completely
determined by $L$.

\begin{example}\label{B3-Einstein}
If the background metric $g$ is Einstein, i.e., if $\overline{\Ric} = \lambda g$, then
\begin{align*}
   12 \B_3 & = - 2 \lo^{ij} \bar{\nabla}_0(\widebar{W})_{0ij0} - 2 H (\lo,\W) + 8 (\lo^2,\W) + 2 |\W|^2\\
   & + 6 (\lo,\Hess(H)) + \Delta (|\lo|^2) + 6 H \tr(\lo^3) + |\lo|^4 + 12 |dH|^2.
\end{align*}
\end{example}

Of course, if in addition $\overline{W}=0$, then this formula reduces to Corollary \ref{B3-inter-corr}.

\begin{proof}
The assumption implies $\bar{\J} = \frac{2}{3} \lambda$, $\bar{\Rho} = \frac{1}{6} \lambda g$, $\bar{\G} = \frac{1}{3} \lambda h + \W$
and $|\bar{\G}|^2  = \frac{1}{3} \lambda^2 + |\W|^2$. Furthermore, the terms in \eqref{B3-g} and the terms in the first line of \eqref{B3-u}
except the second one vanish. The remaining terms in \eqref{B3-u}--\eqref{Weyl} read
$$
   -3 \lambda |\lo|^2 - \lambda |\lo|^2 -2 H(\lo,\W) + \frac{8}{3} \lambda |\lo|^2 + 8 (\lo^2,\W)
   + \frac{4}{3} \lambda |\lo|^2 + 2 |\W|^2.
$$
Simplification proves the claim.
\end{proof}

\subsection{Evaluation II}

We further simplify Lemma \ref{B3-inter2}. The following result shows that, up to
some contributions of the Weyl tensor, the first two terms in \eqref{B3-u} cancel and that
the last term of \eqref{B3-u} cancels against the first term of \eqref{B3-gf}.

\begin{lem}\label{HL1} Let $n=3$. Then it holds
\begin{align*}
   (\lo, \bar{\nabla}_0(\overline{\Ric})) + 2  \lo^{ij} \bar{\nabla}_0(\overline{W})_{0ij0}
   & = 2 \lo^{ij} \bar{\nabla}_0(\bar{R})_{0ij0}, \\
   (\lo,\overline{\Ric}) + 2 (\lo, \W)  & =  2 (\lo,\bar{\G}).
\end{align*}
\end{lem}

\begin{proof} By the Kulkarni-Nomizu decomposition $R = - \Rho \owedge g + W$, we have
$$
   \bar{\G}_{ij} = \bar{R}_{0ij0} =  \bar{\Rho}_{ij} + \bar{\Rho}_{00} (h_r)_{ij} + \overline{W}_{0ij0}.
$$
Hence using
$$
\bar{\nabla}_0(\bar{R})_{0ij0} = \partial_0 (\bar{R}_{0ij0})
   - \bar{R}(\partial_0,\bar{\nabla}_0(\partial_i),\partial_j,\partial_0)
   - \bar{R}(\partial_0,\partial_i,\bar{\nabla}_0(\partial_j),\partial_0)
$$
and $\bar{\nabla}_0(\partial_i) = L_i^k \partial_k$ we find
\begin{align*}
    \bar{\nabla}_0(\bar{R})_{0ij0} - \bar{\nabla}_0(\overline{W})_{0ij0} &
   = \partial_0 (\bar{\Rho}_{ij}) + \partial_0 (\bar{\Rho}_{00}) h_{ij} + \bar{\Rho}_{00} h_{ij}' \\
   & - \bar{\Rho}(\bar{\nabla}_0(\partial_i),\partial_j) - \bar{\Rho}(\partial_i,\bar{\nabla}_0(\partial_j))
   - \bar{\Rho}_{00} h(\bar{\nabla}_0(\partial_i),\partial_j) - \bar{\Rho}_{00} h(\partial_i,\bar{\nabla}_0(\partial_j)) \\
   & = \bar{\nabla}_0(\bar{\Rho})_{ij} + \partial_0 (\bar{\Rho}_{00}) h_{ij} + 2 \bar{\Rho}_{00} L_{ij}
   - 2 \bar{\Rho}_{00} L_{ij} \\
   & = \bar{\nabla}_0(\bar{\Rho})_{ij} + \partial_0 (\bar{\Rho}_{00}) h_{ij}
\end{align*}
for $r=0$. Therefore,
$$
   2 \lo^{ij} \bar{\nabla}_0(\bar{R})_{0ij0} = 2 \lo^{ij} \bar{\nabla}_0(\bar{\Rho})_{ij}
  + 2  \lo^{ij} \bar{\nabla}_0(\overline{W})_{0ij0}
   = \lo^{ij} \bar{\nabla}_0(\overline{\Ric})_{ij} + 2  \lo^{ij} \bar{\nabla}_0(\overline{W})_{0ij0}.
$$
This proves the first identity.
The second identity follows from the decomposition $\bar{\G} = \bar{\Rho} + \bar{\Rho}_{00} h + \W$.
\end{proof}

Next, we evaluate the first term of \eqref{B3-g}.

\begin{lem}\label{del-Nabla} In general dimensions, it holds
\begin{equation}\label{del-nabla}
   \delta (\bar{\nabla}_0(\overline{\Ric})_{0}) = \frac{1}{2} \Delta (\overline{\scal}) - n \delta (H \overline{\Ric}_0)
   - \delta ((L \overline{\Ric})_{0}) - \delta \delta (\overline{\Ric}).
\end{equation}
\end{lem}

\begin{proof} The result follows from the second Bianchi identity. Combining the identity
\begin{equation*}
   \bar{\nabla}_0(\overline{\Ric})(\partial_0,\partial_a)
   = \delta^g (\overline{\Ric})(\partial_a) - h^{ij} \bar{\nabla}_{\partial_i} (\overline{\Ric})(\partial_j,\partial_a)
\end{equation*}
on $M$ with $2 \delta^g (\Ric^g) =d \scal^g$, we obtain
\begin{align*}
    \bar{\nabla}_0(\overline{\Ric})(\partial_0,\partial_a) & = \frac{1}{2} \langle d \overline{\scal},\partial_a \rangle
    - h^{ij} \partial_i (\overline{\Ric}(\partial_j,\partial_a))
    + h^{ij} \overline{\Ric}(\bar{\nabla}_{\partial_i}(\partial_j),\partial_a)
    + h^{ij} \overline{\Ric}(\partial_j,\bar{\nabla}_{\partial_i}(\partial_a)) \\
    & = \frac{1}{2} \langle d \overline{\scal},\partial_a \rangle - h^{ij} \partial_i (\overline{\Ric}(\partial_j,\partial_a)) \\
    & + h^{ij} \overline{\Ric}(\nabla_{\partial_i}(\partial_j) - L_{ij} \partial_0,\partial_a)
    + h^{ij} \overline{\Ric}(\partial_j,\nabla_{\partial_i}(\partial_a) - L_{ia} \partial_0) \\
    & =  \frac{1}{2} \langle d \overline{\scal},\partial_a \rangle - \delta^h (\overline{\Ric})(\partial_a)
    - n H \overline{\Ric}(\partial_0,\partial_a) - h^{ij} L_{ia} \overline{\Ric}(\partial_j,\partial_0).
\end{align*}
Now we apply $\delta = \delta^h$ to this identity of $1$-forms on $M$. We obtain
$$
      \delta (\bar{\nabla}_0(\overline{\Ric})_{0}) = \frac{1}{2} \Delta (\overline{\scal})
      - \delta \delta (\overline{\Ric})
      -n \delta (H \overline{\Ric}_0) - \delta ((L \overline{\Ric})_{0}).
$$
The proof is complete.
\end{proof}

In order to apply Lemma \ref{del-Nabla}, we combine it with the following formula for the last term
on the right-hand side of \eqref{del-nabla}.

\begin{lem}\label{deldel} Let $n=3$. Then
$$
  \delta \delta (\overline{\Ric}) = 2 \Delta (\J) + \Delta (\bar{\J}) - \Delta (H^2) - 2 \delta \delta (H \lo)
 + 2 \delta \delta (\lo^2) - \frac{1}{2} \Delta (|\lo|^2) + 2 \delta \delta (\W).
$$
\end{lem}

\begin{proof} First, we note that
$$
   \delta \delta (\overline{\Ric}) = 2 \delta \delta (\bar{\Rho}) + \delta \delta (\bar{\J} h)
  =  2 \delta \delta (\bar{\Rho}) + \Delta (\bar{\J}).
$$
Now we utilize the identity \eqref{Fial}. It follows that
$$
   2 \delta \delta (\bar{\Rho}) = 2 \delta \delta (\Rho) - 2 \delta \delta (H \lo) - \Delta (H^2)
  + 2 \delta \delta (\lo^2) - \frac{1}{2} \Delta (|\lo|^2) + 2 \delta \delta (\W).
$$
Combining these results with $\delta (\Rho) = d\J$ proves the assertion.
\end{proof}

Now, combining Lemma \ref{del-Nabla} and Lemma \ref{deldel} with the Gauss identity
$$
    \bar{\J} - \J = \bar{\Rho}_{00} + 1/4 |\lo|^2 - 3/2 H^2
$$
gives
\begin{align}\label{surprise}
    \delta (\bar{\nabla}_0(\overline{\Ric})_{0}) & = 2 \Delta (\bar{\J} - \J) - 3 \delta (H \overline{\Ric}_0)
    - \delta ((L \overline{\Ric})_{0}) \notag \\
   & + \Delta(H^2) + 2 \delta \delta (H \lo) - 2 \delta \delta (\lo^2) + \frac{1}{2} \Delta (|\lo|^2)
   - 2 \delta \delta (\W) \notag \\
    & = 2 \Delta (\bar{\Rho}_{00}) + \Delta (|\lo|^2) - 2 \Delta(H^2) - 3 \delta (H \overline{\Ric}_0)
   - \delta ((L \overline{\Ric})_{0}) \notag \\
   & + 2 \delta \delta (H \lo) - 2 \delta \delta (\lo^2) - 2 \delta \delta (\W).
\end{align}
This result shows that \eqref{B3-g} equals
\begin{align}\label{R1}
  & - \Delta (|\lo|^2)  - 2 \delta \delta (H \lo) +  2 \delta \delta (\lo^2) + 2 \Delta(H^2) \notag \\
  & + \delta (H \overline{\Ric}_0) + \delta ((L \overline{\Ric})_{0}) + 6 (dH,\overline{\Ric}_0) + 2 \delta \delta (\W).
\end{align}

\begin{rem}\label{surp2} The identity \eqref{surprise} shows that
$$
   \Delta (|\lo|^2) - 2 \delta \delta (\lo^2) + 2 \delta \delta (H \lo) - 2 \Delta(H^2) = 0
$$
for a flat background. This relation also is a consequence of the difference formula \eqref{basic-div} and
the identity
\begin{equation}\label{dd-flat}
   \delta \delta (H \lo) = (\lo,\Hess(H)) + 4 |dH|^2 + 2 H \Delta(H)
\end{equation}
which is a consequence of Lemma \ref{deldel2} and the Codazzi-Mainardi relation $\delta (\lo) = 2dH$ for a flat background.
More generally, if $g$ is Einstein, i.e., if $\overline{\Ric} = \lambda g$, then $\overline{\Ric}_0 = 0$
and the identity \eqref{surprise} implies
$$
    \Delta (|\lo|^2) - 2 \delta \delta (\lo^2) + 2 \delta \delta (H \lo) - 2 \Delta(H^2)  - 2 \delta \delta (\W) = 0.
$$
By combination with the Codazzi-Mainardi relation $\delta(\lo) = 2 dH$, Lemma \ref{diff-key-g} and Lemma \ref{deldel2},
we conclude the interesting identity
$$
   -2 \lo^{ij} \nabla^k \overline{W}_{kij0}  + |\overline{W}_{0}|^2 - 2 \delta \delta (\W) = 0.
$$
\end{rem}

\begin{lem}\label{deldel2} In general dimensions, it holds
$$
   \delta \delta (H \lo) = (\lo,\Hess(H)) + 2 (dH,\delta(\lo)) + H \delta \delta (\lo).
$$
\end{lem}

\begin{proof}
The identity is obvious.
\end{proof}

Now we combine formula \eqref{R1} with \eqref{B3-u}--\eqref{B3-gc}. Note that the term
$\delta ((L \overline{\Ric})_{0})$ in \eqref{R1} sums up with the fifth term in \eqref{B3-u}
to $\delta (H \overline{\Ric}_0)$ and
that the term  $-\Delta (|\lo|^2)$ in \eqref{R1} cancels with the term $\Delta (|\lo|^2)$ in \eqref{B3-gc}.
We also use Lemma \ref{deldel2}.

By Lemmas \ref{HL1}--\ref{deldel2}, the formula in Lemma \ref{B3-inter2} turns into the sum of
\begin{equation}\label{F1}
   2 \Delta(H^2) + 2 \delta (H \overline{\Ric}_0) \stackrel{!}{=} 2 \delta ( H \delta (\lo)) =
   2 H \delta \delta (\lo) + 2 (dH,\delta(\lo))
\end{equation}
(by Codazzi-Mainardi $\overline{\Ric}_0 = \delta(\lo) - 2dH$),
$$
   6 (dH,\overline{\Ric}_0),
$$
\begin{align}\label{F2}
   & - 2 \delta \delta (H \lo) + 2 \delta\delta (\lo^2) + 2 (\lo, \nabla (\overline{\Ric}_0))
   + 2(\delta(\lo),\overline{\Ric}_0) -3 |\lo|^2 \overline{\Ric}_{00} - (\lo^2,\overline{\Ric}) \notag \\
   & \stackrel{!}{=}
   -2 (\lo,\Hess(H)) - 4 (dH,\delta(\lo)) + 2(\delta(\lo),\overline{\Ric}_0) - 2H \delta \delta (\lo)
   + 2 \delta\delta (\lo^2) \notag \\
   & + 2 (\lo, \nabla (\overline{\Ric}_0)) -3 |\lo|^2 \overline{\Ric}_{00} - (\lo^2,\overline{\Ric})
\end{align}
(by Lemma  \ref{deldel2}),
\begin{align*}
    8 (\lo^2,\bar{\G}) + 2 |\lo|^2 \bar{\J} & = 8 (\lo^2,\bar{\Rho}) + 8  |\lo|^2 \bar{\Rho}_{00}
   + 2 |\lo|^2 \bar{\J} + 8 (\lo^2,\W) & \mbox{(by $\bar{\G} = \bar{\Rho} + \bar{\Rho}_{00} h + \W$)} \\
    & = 4 (\lo^2,\overline{\Ric}) - 2|\lo|^2  \bar{\J} +  8 |\lo|^2\bar{\Rho}_{00} + 8 (\lo^2,\W) \\
    & \stackrel{!}{=} 4 (\lo^2,\overline{\Ric}) - 6 |\lo|^2  \bar{\J} + 4 |\lo|^2 \overline{\Ric}_{00} + 8 (\lo^2,\W),
\end{align*}
\begin{equation}\label{W-terms}
   2 (\bar{\Rho},\W) + 2 |\W|^2 + 2 \delta \delta (\W) - 2 \lo^{ij} \bar{\nabla}_0(\overline{W})_{0ij0}
   -  2H (\lo,\W)
\end{equation}
and
\begin{align*}
    & 6 (\lo,\Hess(H)) +  6 H \tr(\lo^3) + |\lo|^4 + 12 |dH|^2.
\end{align*}

Note that the contributions $2H \delta \delta (\lo)$ in \eqref{F1} and \eqref{F2} cancel.

By Codazzi-Mainardi, we find
$$
   (\delta(\lo),\overline{\Ric}_0) - (\delta(\lo),dH) = |\delta(\lo)|^2 - 3 (\delta(\lo),dH).
$$
Hence
$$
   2(dH,\delta(\lo)) - 4 (dH,\delta(\lo)) + 6 (dH,\overline{\Ric}_0) + 2 (\delta(\lo),\overline{\Ric}_0)
   \stackrel{!}{=} 2 |\delta(\lo)|^2 - 12 |dH|^2.
$$

Thus, we have proved

\begin{prop}\label{B3-main} $12 \B_3$ equals the sum of
\begin{align}\label{B3F1}
    2 \delta \delta (\lo^2) + 2 |\delta(\lo)|^2,
\end{align}
\begin{equation}\label{B3F2}
    2 (\lo, \nabla (\overline{\Ric}_0))  +|\lo|^2 \overline{\Ric}_{00} + 3 (\lo^2,\overline{\Ric}) - 6 |\lo|^2  \bar{\J},
\end{equation}
the Weyl-curvature terms
\begin{equation}\label{W-terms-F}
     2 (\bar{\Rho},\W) + 2 |\W|^2 + 2 \delta \delta (\W) - 2 \lo^{ij} \bar{\nabla}_0(\overline{W})_{0ij0}
    -  2H (\lo,\W)  + 8 (\lo^2,\W)
\end{equation}
and
\begin{equation}\label{B3F3}
   4 (\lo,\Hess(H)) +  6 H \tr(\lo^3) + |\lo|^4.
\end{equation}
\end{prop}

Note that there is no Laplace term in that formula.

\subsection{Proof of the main result. Equivalences}\label{equiv}


Proposition \ref{B3-main} has the disadvantage that the conformal invariance of $\B_3$ is not obvious. Therefore, 
it is natural to reformulate the results in a way which makes the conformal invariance transparent. For this purpose, 
we relate Proposition \ref{B3-main} to the formula
\begin{align}\label{B3-GW}
   12 \B_3 & = 6 \LOP ((\lo^2)_\circ)  + 2 |\lo|^4  + \star  \notag \\
   & = 6 \delta \delta (\lo^2) - 2 \Delta (|\lo|^2) + 6 (\lo^2,\Rho) - 2 |\lo|^2 \J + 2 |\lo|^4 + \star
\end{align}
in \cite[Proposition 1.1]{GGHW}, where\footnote{Here we refer to \url{arXiv:1508.01838v1}. This result differs from the
version in the published paper.}
\begin{equation}\label{GGHW-terms}
    \star \st  2 \LOP (\W) + 4 |\W|^2 + 2 |\overline{W}_{0}|^2 - 2 (\lo,B) + 14 (\lo^2,\W) - 2 \lo^{ab} \lo^{cd} \overline{W}_{cabd}.
\end{equation}
Here $B$ is a certain conformally invariant symmetric bilinear form of weight $-1$ which will be defined in 
\eqref{Bach-def}. Here we took into account that in \cite{GGHW} the signs of the components of the curvature 
tensor and the Weyl tensor are opposite to ours. All terms in \eqref{GGHW-terms} are conformally invariant. The 
difference of both formulas is
\begin{align}\label{diff-g}
   &  2 ( \Delta (|\lo|^2) - 2 \delta \delta (\lo^2)) + 2 |\delta(\lo)|^2 \notag \\
   &  + 4 (\lo,\Hess(H)) +  6 H \tr(\lo^3) - |\lo|^4 \notag \\
   & + 2 (\lo, \nabla (\overline{\Ric}_0)) + |\lo|^2 \overline{\Ric}_{00}
   + 3 (\lo^2,\overline{\Ric})  - 6 |\lo|^2  \bar{\J} - 6 (\lo^2,\Rho) + 2 |\lo|^2 \J \notag \\
   &  + 2 (\bar{\Rho},\W) + 2 |\W|^2 + 2 \delta \delta (\W) - 2 \lo^{ij} \bar{\nabla}_0(\overline{W})_{0ij0}
    -  2H (\lo,\W)  + 8 (\lo^2,\W) - \star.
\end{align}

\begin{lem}\label{van-equiv} The sum \eqref{diff-g} vanishes.
\end{lem}

In other words, Proposition \ref{B3-main} is equivalent to \cite[Proposition 1.1]{GGHW}. The proof of this result will
also establish the equivalence to Theorem \ref{main1}.

\begin{rem}\label{equiv-van}
Lemma \ref{van-equiv} holds for a flat background metric.  In this case  $\star = 0$. In fact, the identity \eqref{Fial} implies
\begin{equation}\label{JP}
    6(\lo^2,\Rho) - 2 \J |\lo|^2  =  6 H \tr(\lo^3) - |\lo|^4.
\end{equation}
By $\delta(\lo) = 2 dH$ (Codazzi-Mainardi), the sum \eqref{diff-g} equals
\begin{align*}
   & 2 ( \Delta (|\lo|^2) - 2 \delta \delta (\lo^2)) + 8 |dH|^2 \\
   & + 4 (\lo,\Hess(H)) + 6 H \tr (\lo^3) - |\lo|^4 - 6 H \tr(\lo^3) + |\lo|^4 \\
   & =  2 ( \Delta (|\lo|^2) - 2 \delta \delta (\lo^2))  + 8 |dH|^2 + 4 (\lo,\Hess(H)).
\end{align*}
The identity \eqref{basic-div} shows that this sum vanishes.
\end{rem}

A key role in the argument in Remark \ref{equiv-van} is played by the formula \eqref{basic-div} for the divergence term
$\Delta (|\lo|^2) - 2 \delta \delta (\lo^2)$. Lemma \ref{diff-key-g} extends this result to general backgrounds. We also
need the following curved analog of \eqref{JP}.

\begin{lem}\label{FH} If $n=3$, then it holds
$$
    6(\lo^2,\Rho) - 2 |\lo|^2 \J =  6 H \tr(\lo^3) - |\lo|^4 + 6 (\lo^2,\bar{\Rho})
   - 2 |\lo|^2 \bar{\J} + 2 |\lo|^2 \bar{\Rho}_{00} -  6(\lo^2,\W).
$$
\end{lem}

\begin{proof} The identity \eqref{Fial} yields
$$
   \iota^* \bar{\Rho} - \Rho =  \lo^2 - \frac{1}{4} |\lo|^2 h - H \lo - \frac{1}{2} H^2 h + \W.
$$
Taking the trace yields the Gauss identity
$$
    \bar{\J} - \bar{\Rho}_{00} - \J = |\lo|^2 -\frac{3}{4} |\lo|^2 - \frac{3}{2} H^2
   = \frac{1}{4}  |\lo|^2 - \frac{3}{2} H^2.
$$
These relations imply the assertion.
\end{proof}

Now, by Lemma \ref{FH}, \eqref{diff-g} simplifies to
\begin{align}\label{diff-h}
   &  2 ( \Delta (|\lo|^2) - 2 \delta \delta (\lo^2)) + 2 |\delta(\lo)|^2 + 4 (\lo,\Hess(H)) \notag \\
   & + 2 (\lo, \nabla (\overline{\Ric}_0)) + |\lo|^2 \overline{\Ric}_{00}
   + 3 (\lo^2,\overline{\Ric}) - 4 |\lo|^2 \bar{\J} -6 (\lo^2,\bar{\Rho}) - 2 |\lo|^2 \bar{\Rho}_{00} \notag \\
   & + 2 (\bar{\Rho},\W) + 2 |\W|^2 + 2 \delta \delta (\W) - 2 \lo^{ij} \bar{\nabla}_0(\overline{W})_{0ij0}
    -  2H (\lo,\W)  + 14  (\lo^2,\W) - \star
\end{align}
But
$$
   |\lo|^2 \overline{\Ric}_{00} + 3 (\lo^2,\overline{\Ric})  - 4 |\lo|^2 \bar{\J} - 6 (\lo^2,\bar{\Rho})
  - 2 |\lo|^2\bar{\Rho}_{00} = 0,
$$
i.e., the second last line of \eqref{diff-h} reduces to $2 (\lo, \nabla (\overline{\Ric}_0))$. Therefore, \eqref{diff-g}
further simplifies to
\begin{align*}
   &  2 ( \Delta (|\lo|^2) - 2 \delta \delta (\lo^2)) + 2 |\delta(\lo)|^2
   + 4 (\lo,\Hess(H))  + 2 (\lo, \nabla (\overline{\Ric}_0)) \\
   & + 2 (\bar{\Rho},\W) + 2 |\W|^2 + 2 \delta \delta (\W) - 2 \lo^{ij} \bar{\nabla}_0(\overline{W})_{0ij0}
    -  2H (\lo,\W)  + 14 (\lo^2,\W) - \star.
\end{align*}
Now, by Lemma \ref{diff-key-g}, this sum equals
\begin{align}\label{a-full}
   & 4 \lo^{ij} \nabla^k \overline{W}_{ikj0}  + 2 |\overline{W}_{0}|^2 \notag \\
   & + 2 (\bar{\Rho},\W)  + 2 |\W|^2 + 2 \delta \delta (\W) - 2 \lo^{ij} \bar{\nabla}_0(\overline{W})_{0ij0}
    -  2 H (\lo,\W)  + 14 (\lo^2,\W) - \star.
\end{align}
Now we apply the identity
$$
    (\bar{\Rho},\W) = (\Rho,\W) + (\lo^2,\W) - H (\lo,\W) + |\W|^2
$$
(see \eqref{Fial}). Hence  the sum \eqref{a-full} equals
\begin{align*}
    & 2 \delta \delta (\W) + 2 (\Rho,\W) - 2 \lo^{ij} \bar{\nabla}_0(\overline{W})_{0ij0} - 4 H(\lo,\W) \\
    & + 16 (\lo^2,\W)  + 4 |\W|^2 + 2 |\overline{W}_{0}|^2 + 4 \lo^{ij} \nabla^k \overline{W}_{ikj0} - \star.
\end{align*}
Therefore, Lemma \ref{van-equiv} holds true iff
\begin{align}\label{star}
    \star & = 2 \LOP(\W) - 2 \lo^{ij} \bar{\nabla}_0(\overline{W})_{0ij0} + 4 \lo^{ij} \nabla^k \overline{W}_{ikj0} \notag \\
    & - 4 H(\lo,\W) + 16 (\lo^2,\W)  + 4 |\W|^2 + 2 |\overline{W}_{0}|^2.
\end{align}
Equivalently, Lemma \ref{van-equiv} holds true iff
\begin{align}\label{final}
    & -2 (\lo,B) + 14 (\lo^2,\W) - 2 \lo^{ij} \lo^{kl} \overline{W}_{kijl} \notag \\
    & = - 2 \lo^{ij} \bar{\nabla}_0(\overline{W})_{0ij0} + 4 \lo^{ij} \nabla^k \overline{W}_{ikj0} - 4 H(\lo,\W) + 16 (\lo^2,\W)
\end{align}
It remains to prove \eqref{final}. As a preparation, we observe

\begin{lem}\label{LW-T} Let $n=3$. Then
$$
    \lo^{ij} \bar{\nabla}^k \overline{W}_{ikj0} = \lo^{ij} \nabla^k \overline{W}_{ikj0} + (\lo^2,\W)
   - 3 H(\lo,\W) + \lo^{ij} \lo^{kl} \overline{W}_{kijl}.
$$
\end{lem}

\begin{proof}
By $\bar{\nabla}_i (\partial_j) = \nabla_i(\partial_j) - L_{ij} \partial_0$
and $\bar{\nabla}_k(\partial_0) = L_k^m \partial_m$, we find
\begin{align*}
    \bar{\nabla}^k \overline{W}_{ikj0} & = \nabla^k \overline{W}_{ikj0} - L^{kl} \overline{W}_{ikjl}
   + L^k_i \overline{W}_{0kj0} + 3 H \overline{W}_{i0j0} + L^k_j \overline{W}_{ik00} \\
   & = \nabla^k \overline{W}_{ikj0} + L^{kl} \overline{W}_{kijl} + L^k_i \overline{W}_{0kj0} - 3 H \overline{W}_{0ij0} \\
   & = \nabla^k \overline{W}_{ikj0} + \lo^{kl} \overline{W}_{kijl} - H \overline{W}_{0ij0} + \lo_i^k \overline{W}_{0kj0}
   + H \overline{W}_{0ij0} - 3 H \overline{W}_{0ij0} \\
   & =  \nabla^k \overline{W}_{ikj0} + \lo^{kl} \overline{W}_{kijl} + \lo_i^k \overline{W}_{0kj0} - 3 H \overline{W}_{0ij0}.
\end{align*}
The assertion follows by contraction with $\lo^{ij}$.
\end{proof}

\begin{lem}\label{Bach-relation} It holds
\begin{equation}\label{Bach-deco}
   (\lo,B) = \lo^{ij} \bar{\nabla}^0 (\widebar{W})_{0ij0} + 2 H (\lo,\W)
   - 2 \lo^{ij} \nabla^k \overline{W}_{jki0} - (\lo^2,\W) - \lo^{ij}\lo^{kl} \overline{W}_{kijl}.
\end{equation}
\end{lem}

\index{$B$ \quad hypersurface Bach tensor}

\begin{proof} We first restate the definition of $B$ in our conventions:\footnote{We recall that our signs of the 
components of $\overline{W}$ are opposite.}
\begin{align}\label{Bach-def}
    B_{ij} & = \bar{C}_{0(ij)} - H \W_{ij} + \nabla^k \overline{W}_{0(ij)k} \\
    & = \bar{\nabla}^k (\overline{W})_{0(ij)k} - H \W_{ij} + \nabla^k \overline{W}_{0(ij)k}. \notag
\end{align}
Here                  \index{$C$ \quad Cotton tensor}
$$
   (n-3) C_{ijk} \st \nabla^l (W)_{ijkl}
$$
defines the Cotton tensor $C$ on a manifold of dimension $n$. It satisfies the conformal transformation law
$$
   \hat{C}_{ijk} = C_{ijk} + W_{ijk \grad(\varphi)}.
$$
We emphasize that, in the definition \eqref{Bach-def} of $B$, the index $k$ in the first term runs over the 
tangential {\em and} the normal vectors. $(ij)$ denotes symmetrization. We first verify that the symmetric 
tensor $B$ satisfies the conformal transformtion law $e^{\varphi} \hat{B} = B$. For this purpose, we first 
observe that
$$
   e^\varphi \hat{\bar{C}}_{\hat{0}ij} = \hat{\bar{C}}_{0ij} 
  = \bar{C}_{0ij} + \overline{W}_{0ij\grad (\varphi)} =  \bar{C}_{0ij}
  + \overline{W}_{0ij\grad^t (\varphi)} + \W_{ij}  \partial_0(\varphi),
$$
where $\hat{0} = \hat{\partial}_0 = e^{-\varphi} \partial_0$ and 
$\grad^t(\varphi)$ is the tangential component of the gradient. We also recall that
$$
   e^{-2\varphi} \hat{\widebar{W}}_{ijkl} = \widebar{W}_{ijkl}, \quad \hat{\bar{\W}}_{ij} = \bar{\W}_{ij}
   \quad \mbox{and} \quad e^\varphi \hat{H} = H + \partial_0(\varphi).
$$
We calculate
\begin{align*}
    e^{2\varphi} \hat{\nabla}^k \hat{\widebar{W}}_{\hat{0}ijk} & = \partial^k (e^\varphi \widebar{W}_{0ijk}) \\
    & - e^{\varphi} \widebar{W}(\partial_0,\hat{\nabla}^k(\partial_i),\partial_j,\partial_k)
    - e^{\varphi}  \widebar{W}(\partial_0,\partial_i, \hat{\nabla}^k(\partial_j),\partial_k)
     - e^{\varphi}  \widebar{W}(\partial_0,\partial_i, \partial_j, \hat{\nabla}^k(\partial_k)).
\end{align*}
Now the general transformation law
$$
   \hat{\nabla}_i(\partial_j) = \nabla_i(\partial_j) + \partial_i(\varphi) \partial_j 
   + \partial_j(\varphi) \partial_i - g_{ij} \grad(\varphi)
$$
implies that in general dimensions ($\dim (M) = n$)\footnote{That identity corrects \cite[(2.10)]{GGHW}.}
$$
    e^{\varphi} \hat{\nabla}^k \hat{\widebar{W}}_{\hat{0}ijk}
   = \nabla^k \widebar{W}_{0ijk} + (n-4) \widebar{W}_{0ij\grad^t(\varphi)}
    - \widebar{W}_{0\grad^t(\varphi)ij}.
$$
Hence for $n=3$ we find
$$
     e^{\varphi} \hat{\nabla}^k \hat{\widebar{W}}_{\hat{0}(ij)k}
    = \nabla^k \widebar{W}_{0(ij)k} - \widebar{W}_{0(ij)\grad^t(\varphi)}.
$$
Therefore,
$$
   e^\varphi \hat{B}_{ij} = B_{ij} + \overline{W}_{0(ij)\grad^t (\varphi)} + \W_{ij}  \partial_0(\varphi) -  \W_{ij}  \partial_0(\varphi)
   - \widebar{W}_{0(ij)\grad^t(\varphi)} = B_{ij}.
$$
Now the definition of $B$ gives
\begin{align*}
   (\lo,B) & \st \lo^{ij} \bar{\nabla}^k (\overline{W})_{0ijk} - H (\lo,\W) + \lo^{ij} \nabla^k \overline{W}_{0ijk} \\
   & = - \lo^{ij} \bar{\nabla}^k (\overline{W})^t_{jki0} + \lo^{ij} \bar{\nabla}^0 (\overline{W})_{0ij0} - H (\lo,\W)
   - \lo^{ij} \nabla^k \overline{W}_{jki0},
\end{align*}
where the superscript $t$ in the first sum indicates that indices are only tangential. Now Lemma \ref{LW-T} yields
\begin{align*}
   (\lo,B) & = \lo^{ij} \bar{\nabla}^0 (\overline{W})_{0ji0} + 2 H(\lo,\W) - 2 \lo^{ij} \nabla^k \overline{W}_{jki0} 
   - (\lo^2,\W) - \lo^{ij} \lo^{kl} \overline{W}_{kijl} .
\end{align*}
The proof of \eqref{Bach-deco} is complete.
\end{proof}

Lemma \ref{Bach-relation} shows that 
$$
    - 2(\lo,B) = -2 \lo^{ij} \bar{\nabla}^0 (\widebar{W})_{0ij0} - 4  H (\lo,\W)
   +4 \lo^{ij} \nabla^k \overline{W}_{jki0} + 2 (\lo^2,\W) +2 \lo^{ij}\lo^{kl} \overline{W}_{kijl}.
$$
Hence
$$
 -2 (\lo,B) + 14 (\lo^2,\W) - 2 \lo^{ij} \lo^{kl} \overline{W}_{kijl}
   = 16 (\lo^2,\W)  -2 \bar{\nabla}^0 (\widebar{W})_{0ij0} -4 H \W_{ij} + 4 \lo^{ij} \nabla^k \overline{W}_{jki0}
$$
(note the cancellation!). This proves \eqref{final} and hence Lemma \ref{van-equiv}.

Now in order to finish the {\bf proof of Theorem \ref{main1}}, it suffices to combine \eqref{B3-GW} 
with \eqref{star}.

\begin{rem}\label{GGHW-wrong} The formula for $\B_3$ in the published version of \cite{GGHW} reads
$$
   12 \B_3 = 4 \LOP ((\lo^2)_\circ) + 2 \LOP (\Fo) - 2 (\lo,B) + |\lo|^4 + 4 (\Fo,\JF) + 2(\Fo,\lo^2) + 2|\overline{W}_0|^2.
$$
By $\Fo = (\lo^2)_\circ + \W$, this formula is equivalent to
\begin{align*}
   12 \B_3 & = 6 \LOP((\lo^2)_\circ) + 2 \LOP(\W) - 2 (\lo,B) + |\lo|^4 + 2|\overline{W}_0|^2 \\
   & + 2 ((\lo^2)_\circ + \W,\lo^2) + 4 ((\lo^2)_\circ + \W, \lo^2 - \frac{1}{4} |\lo|^2 h + \W).
\end{align*}
The second line simplifies to
$$
   |\lo|^4 + 10 (\lo^2,\W) + 4 |\W|^2.
$$
If follows that the resulting formula for $12 \B_3$ differs from \eqref{B3-GW}, \eqref{GGHW-terms}.
\end{rem}

Finally, we note that for a conformally flat background Theorem \ref{main1} states that
$$
   6 \B_3 = 3 \LOP ((\lo^2)_\circ) + |\lo|^4 = 3 \delta \delta ((\lo^2)_\circ) + 3 ((\lo^2)_\circ,\Rho) + |\lo|^4.
$$
Lemma \ref{NEW3a} implies that this formula is equivalent to 
\begin{equation}\label{origin}
   6 \B_3  = \Delta (|\lo|^2) - |\nabla \lo|^2 + 3/2 |\delta(\lo)|^2 - 2 \J |\lo|^2 + |\lo|^4
\end{equation}
(as also stated in \cite[Proposition 2.10]{GW-LNY}).

\section{Variational aspects}\label{var}

\index{$\var$ \quad variation}  \index{$\mathcal{W}_3$ \quad higher Willmore functional}

Let $\iota: M^3 \hookrightarrow X^4$ be an embedding. In this section, we prove that the conformally invariant equation
$\B_3 = 0$ is the Euler-Lagrange equation of the conformally invariant functional
\begin{equation}\label{W3}
   \mathcal{W}_3(\iota) \st \int_{\iota(M)} (\tr(\lo^3) + (\lo,\W)) dvol
\end{equation}               
under normal variations of the embedding $\iota$. Let $u \in C^\infty(M)$ and $\partial_0$ be
a unit normal field of $M$. We set $\iota_t (m) = \exp (t u(m) \partial_0)$, where $\exp$ is the exponential map.
Then $\iota_0 = \iota$ and $\iota_t$ is a variation of $M$ with variation field $u \partial_0$. Let 
$\mathcal{W}_3(\iota_t)$ be the analogous functional for $\iota_t$ and define                              
$$
    \var (\mathcal{W}_3)[u] \st (d/dt)|_0 (\mathcal{W}_3(\iota_t)).
$$

\begin{thm}\label{variation} It holds
\begin{equation*}
   -\var(\mathcal{W}_3)[u] = 6 \int_M u \B_3 dvol.
\end{equation*}
\end{thm}

This result reproves \cite[Proposition 1.2]{GGHW}. Our arguments are classical and differ 
substantially from those in the reference (see the comments after the proof).

\begin{proof} We first note that the variation of $\int_M \tr (\lo^3) dvol$ has been determined
in \cite[Lemma 13.9.1]{JO} for conformally flat backgrounds. The given arguments easily extend to the
general case and yield
\begin{align*}
    - \var\left(\int_M \tr (\lo^3) dvol_h\right)[u] & = 3 \int_M u \left(\delta \delta ((\lo^2)_\circ) + (\Rho,(\lo^2)_\circ)
   + 2 (\lo^2,\W) + \frac{1}{3} |\lo|^4\right) dvol_h \\
   & = \int_M u \left( 3 \LOP ((\lo^2)_\circ) + 6 (\lo^2,\W) + |\lo|^4\right) dvol_h.
\end{align*}
In the second part of the proof, we determine the variation of $\int_M (\lo,\W) dvol_h$. We write the integrand as
$h^{ai} h^{bj} \lo_{ab} \W_{ij}$ and apply the well-known variation formulas \cite[Theorem 3-15]{A}, 
\cite[Theorem 3.2]{HP}
\begin{align*}
    \var (h)[u] & = 2 u L, \\
    \var (L)[u] & = - \Hess(u) + u L^2 - u \bar{\G}, \\
    3 \var (H)[u] & = - \Delta (u) - u |L|^2 - u \overline{\Ric}_{00}
\end{align*}
and
$$
   \var(dvol_h)[u] = 3 u H dvol_h.
$$
It follows that the variation is given by the integral of the sum of
\begin{equation}\label{V1}
   - 4u (\lo^2,\W) - 4 u H (\lo, \W) \qquad \mbox{(by variation of the metric)},
\end{equation}
\begin{align}\label{V2}
    (\var(\lo)[u],\W) & = \W^{ij} (\var(L)_{ij} - H \var(h)_{ij}) \notag \\
    & = \W^{ij} (-\Hess_{ij}(u) + u (L^2)_{ij} - u \bar{\G}_{ij}) - 2 u H \W^{ij} L_{ij} \notag \\
    & = - (\Hess (u), \W) + u (\lo^2,\W) - u (\bar{\G},\W),
\end{align}
\begin{align}\label{V3}
   (\lo,\var(\W)[u]) &  = u \lo^{ij} \bar{\nabla}_0 (\overline{W})_{0ij0}
   + u \lo^{ij} (L_i^k \overline{W}_{0kj0} + L^{k}_j \overline{W}_{0ik0}) \notag \\
   & = u  \lo^{ij} \bar{\nabla}_0 (\overline{W})_{0ij0}  + 2 u (\lo^2,\W) + 2 u H(\lo,\W)
\end{align}
and
\begin{align}\label{V5}
   & -2 \lo^{ij} \overline{W}_{\grad(u) ij0} \qquad \mbox{(by variation of the normal vector)}, \notag \\
   & 3 u H  (\lo,\W)  \qquad \mbox{(by variation of the volume form)}.
\end{align}
Now using partial integration we obtain
\begin{align*}
   & \var\left(\int_M (\lo,\W) dvol_h \right)[u] \\
   & = \int_M u \left[ - \delta \delta (\W) - (\lo^2,\W) + H (\lo,\W) - (\bar{\G},\W)
   + \lo^{ij} \bar{\nabla}_0 (\overline{W})_{0ij0} \right] dvol_h \\
   & - 2 \int_M \lo^{ij} \overline{W}_{\grad(u)ij0} dvol_h.
\end{align*}
Since
$$
   (\bar{\G},\W) = (\bar{\Rho},\W) + (\W,\W) = (\Rho,\W) - H (\lo,\W) + (\lo^2,\W) + 2 |\W|^2
$$
by \eqref{Fial}, we get
\begin{align*}
   & \var\left(\int_M (\lo,\W) dvol_h \right)[u] \\
   & = \int_M u \left[ - \LOP(\W) + 2H (\lo,\W) -  2 (\lo^2,\W) - 2 |\W|^2
   + \lo^{ij} \bar{\nabla}_0 (\overline{W})_{0ij0} \right] dvol_h \\
   & - 2 \int_M (du,\lo^{ij} \overline{W}_{\cdot ij0}) dvol_h.
\end{align*}
By partial integration, we find
\begin{align*}
  2 \int_M (du,\lo^{ij} \overline{W}_{\cdot ij0}) dvol_h
  & = - 2 \int_M u \delta (\lo^{ij} \overline{W}_{\cdot ij0}) dvol_h  = - 2 \int_M u \nabla^k (\lo^{ij} \overline{W}_{kij0}) dvol_h \\
  &  = - 2 \int_M u \lo^{ij} \nabla^k \overline{W}_{kij0} dvol_h + \int_M u |\overline{W}_0|^2 dvol_h
\end{align*}
using the trace-free Codazzi-Mainardi equation \eqref{tf-CM} (similarly as on page \pageref{total}). Summarizing these results,
proves the claim.
\end{proof}

Graham's theorem \cite[Theorem 3.1]{Graham-Yamabe} and \cite[(3.8)]{JO} imply that the variation 
of $\int v_3 dvol$ equals $4 \int u \B_3 dvol$. Here the singular Yamabe renormalized volume coefficient $v_3$ 
(as defined in \cite{Graham-Yamabe}) satisfies
$$
   12 \int_Mv_3 dvol = - \int_M \mathbf{Q}_3 dvol = - 8 \int_M (\tr(\lo^3) + (\lo,\W)) dvol = - 8 \mathcal{W}_3,
$$ 
where ${\bf{Q}}_3$ is the extrinsic $Q$-curvature (see \cite[Example 13.10.2 and (13.10.7)]{JO}). In other words, the 
variation of $\mathcal{W}_3$ equals $-6 \int u \B_3 dvol$. This shows that Theorem \ref{variation} 
fits with Graham's theorem. On the other hand, \cite[Proposition 1.2]{GGHW} states that the variation of 
$\mathcal{W}_3$ equals $6 \int u \B_3 dvol$. The discrepancy of the sign is due to the altered definition of variations.
Note that the notion of variation exploited in \cite{GGHW} leads to the result $\var (L)[u] = \Hess (u)$ 
(see \cite[(3.8)]{GGHW}). This formula differs from the usual formula used above.


\printindex

\end{document}